\documentclass[openacc]{rsproca_new}

\newtheorem{theorem}{\bf Theorem}[section]

\newtheorem{corollary}{\bf Corollary}[section]

\newtheorem{lemma}{\bf Lemma}[section]

\usepackage{enumitem}
\usepackage{amsthm}
\usepackage{amsmath}
\usepackage{amssymb}
\usepackage{mathrsfs}
\usepackage{float}
\allowdisplaybreaks

\DeclareMathOperator{\divergence}{div}
\DeclareMathOperator{\grad}{grad}
\DeclareMathOperator{\im}{Im}

\newcommand\scaleddot{\scalebox{.89}{.}}
\makeatletter
\renewcommand{\dddot}[1]{%
  {\mathop{\kern\z@#1}\limits^{\makebox[0pt][c]{\vbox to-2\ex@{\kern-\tw@\ex@\hbox{\normalfont\scaleddot\kern-0.5pt\scaleddot\kern-0.5pt\scaleddot}\vss}}}}}

\begin{document}

\title{Small-amplitude static periodic patterns at a fluid-ferrofluid interface}

\author{
M. D. Groves$^{1,2}$ and J. Horn$^{1}$}

\address{$^{1}$ Fachrichtung Mathematik, Universit\"at des Saarlandes, Postfach 151150, 66041 Saarbr\"ucken, Germany\\
$^{2}$Department of Mathematical Sciences, Loughborough University, Loughborough, LE11 3TU, UK\\}

\subject{differential equations, fluid mechanics}


\corres{M. D. Groves\\
\email{groves@math.uni-sb.de}}

\begin{abstract}
We establish the existence of static doubly periodic patterns (in particular rolls, rectangles and hexagons)
on the free surface of a ferrofluid near onset of the Rosensweig instability, assuming a general (nonlinear)
magnetisation law. A novel formulation of the ferrohydrostatic equations in terms of Dirichlet-Neumann operators
for nonlinear elliptic boundary-value problems is presented. We demonstrate the analyticity of these operators in
suitable function spaces and solve the ferrohydrostatic problem using an analytic version of Crandall-Rabinowitz
local bifurcation theory. Criteria are derived for the bifurcations to be sub-, super- or transcritical
with respect to a dimensionless physical parameter.
\end{abstract}

\begin{fmtext}
\section{Introduction}
Consider two static immiscible perfect fluids in the regions
\begin{align*}
\Omega^\prime & :=\left\{ (x,y,z):\,\eta(x,z)<y<d\right\}, \\
\Omega: & =\left\{ (x,y,z):\,-d<y<\eta(x,z)\right\}
\end{align*}
separated by the free surface $\{y=\eta(x,z)\}$, where gravity acts in the negative $y$ direction.
The upper fluid is non-magnetisable, while the lower is a ferrofluid with a general nonlinear
magnetisation law
$${\mathbf M}=\mathbf{M}(\mathbf{H})=m(|\mathbf{H}|)\frac{\mathbf{H}}{|\mathbf{H}|}$$
expressing the relationship between the magnetisation
${\mathbf M}$ of the ferrofluid and the strength of the magnetic field ${\mathbf H}$.
Subjecting the fluids to a vertically directed magnetic field of sufficient strength leads to the emergence of
interfacial patterns (see Figure \ref{fig: patterns}).
\end{fmtext}

\maketitle

\begin{figure}[H] 
\centering\includegraphics[width=2.8in]{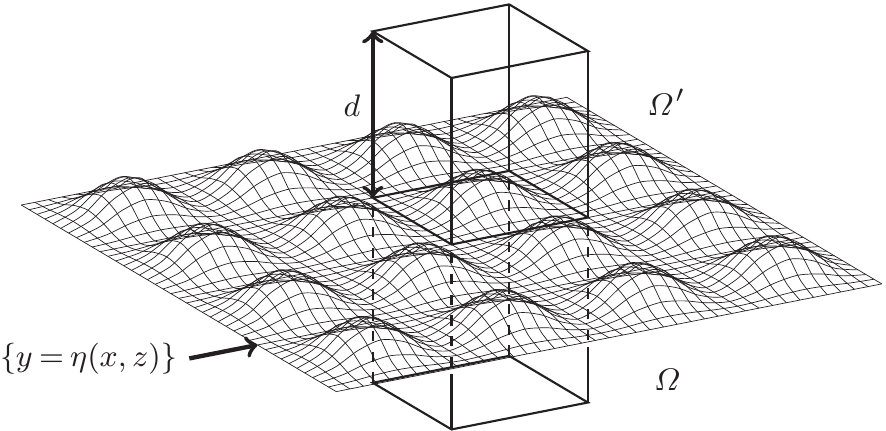} 
\caption{\label{fig: patterns}
A rectangular periodic pattern at the interface $\left\{y=\eta(x,z)\right\}$.} 
\end{figure}

In this article we present an existence theory for small amplitude, doubly periodic patterns with 
\[
\eta(\mathbf{x} + \mathbf{l})=\eta(\mathbf{x})
\]
for every $\mathbf{l}\in \mathscr{L},$ where $\mathbf{x}=(x,z)$ and $\mathscr{L}$ is the lattice given by
$$
\mathscr{L}=\left\{  m\mathbf{l}_{1}+n\mathbf{l}_{2}:\ m,n\in\mathbb{Z} \right\}
$$
with $\left| \mathbf{l}_{1}\right|=\left| \mathbf{l}_{2}\right|. $ 
We are especially interested in three patterns which are observed in experiments (Figure \ref{fig: the Patterns}),
namely rolls, rectangles and hexagons (see Figure \ref{fig: lattices}).

\begin{enumerate}
\item\vspace{-0.8\baselineskip}
For rolls we seek  functions that are independent of the $z$-direction   
and choose \linebreak$\mathbf{l}=(\frac{2\pi}{\omega},0)$, so that
the periodic base cell  
is given by 
$\left\{  x:|x|<\frac{\pi}{\omega} \right\}$.
\item
For rectangles we choose $\mathbf{l}_{1}=(\frac{2\pi}{\omega},0)$, $\mathbf{l}_{2}=(0,\frac{2\pi}{\omega})$,
so that the periodic base cell
is given by 
$\left\{  (x,z):\left|  x\right|,\:\left|  z\right|<\frac{\pi}{\omega} \right\}$.
\item
For hexagons we choose $\mathbf{l}_{1}=\frac{2\pi}{\omega}(1,-\frac{1}{\sqrt{3}})$, 
$\mathbf{l}_{2}=\frac{2\pi}{\omega}(0,\frac{2}{\sqrt{3}}),$ 
so that we obtain an additional periodic direction 
$\mathbf{l}_{3}=\mathbf{l}_{1}+\mathbf{l}_{2}=\frac{2\pi}{\omega}(1,\frac{1}{\sqrt{3}})$  
and the periodic base cell  
is given by 
$\left\{  
(x,z):\left|  x\right|<\frac{2\pi}{\omega},\:
\left| x-\sqrt{3}z\right|<\frac{4\pi}{\omega},
\left| x+\sqrt{3}z\right|<\frac{4\pi}{\omega}
\right\}$.
\end{enumerate}\vspace{-0.8\baselineskip}
Notice that each of these patterns exhibits a rotational symmetry: the shape of the free surface
is invariant under a rotation of the $(x,z)$-plane through respectively (i) $\frac{\pi}{2}$, (ii) $\frac{\pi}{3}$ and (iii) ${\pi}$.

\begin{figure}[H]
\centering\includegraphics[width=3.3in]{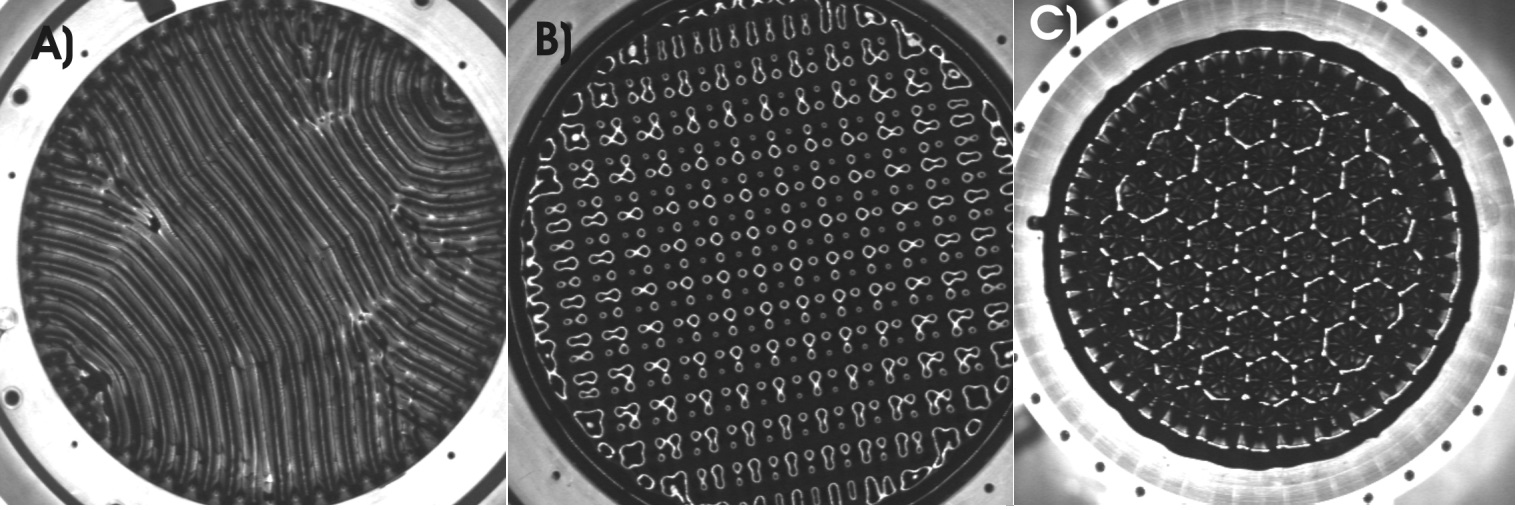} 
\caption{\label{fig: the Patterns}
Experimental observation of static patterns at the fluid-ferrofluid interface:
A) rolls; B) rectangles; C) hexagons (Fachrichtung Physik, Universit\"{a}t des Saarlandes).}  
\end{figure}

In Section \ref{ch: Mathematical description}\ref{sec: The physical problem} 
we derive the mathematical formulation of this problem from physical principles.
The governing equations (equations \eqref{eq: NLAPLACEO}--\eqref{eq: ENBC3}) are formulated
in terms of perturbations $\phi^\prime$ and $\phi$ of magnetic potentials corresponding 
to a uniform vertically directed magnetic field (of strength $h\mu(h)$ in the upper fluid and
$h$ in the lower fluid, where
$\mu$ is obtained from the magnetisation law by the formula
$\mu(s)=1+m(s)/s$; the potentials
are horizontally doubly periodic, satisfying
\[
\phi^\prime(\mathbf{x} + \mathbf{l},y)=\phi^\prime(\mathbf{x},y)
,
\qquad
\phi(\mathbf{x} + \mathbf{l},y)=\phi(\mathbf{x},y)
\]
for every $\mathbf{l}\in \mathscr{L}$ (with a slight abuse of notation).

This problem was first studied by Cowley and Rosensweig \cite{CowleyRosensweig67}. 
Using a linear stability analysis, they found that, 
as the strength $h$ of the magnetic field 
exceeds a critical value $h_{\mathrm{c}}$, the flat surface destabilises and  a hexagonal pattern of peaks appears. 
This phenomenon is known as the \emph{Rosensweig instability}. 
A mathematically rigorous treatment of the problem was given by Twombly and Thomas \cite{TwomblyThomas83}, who 
used coordinate transformations to `flatten' the free surface by
transforming  the \emph{a priori} unknown domains $\Omega^\prime$ and $\Omega$ into fixed strips. Applying
Lyapunov-Schmidt reduction reduces these transformed equations for rotationally symmetric patterns (see below)
to a locally equivalent one-dimensional equation 
which is solved using the implicit-function theorem; 
the result is the existence, for values of $h$ near $h_{\mathrm{c}}$, of rolls and rectangles in addition to the hexagonal pattern.    
Twombly and Thomas's work is however flawed by some miscalculations and mathematical inconsistencies, and is also restricted to
linear magnetisation laws. In this article we present a more systematic approach 
which is motivated by the corresponding study of doubly periodic travelling water waves 
by Craig and Nicholls \cite{CraigNicholls00}; we also consider general nonlinear magnetisation laws.

We work with dimensionless variables,  in terms of which the problem depends upon
two dimensionless parameters $\beta$ (whose value $\beta_0$ is fixed) and $\gamma$ (see equation \eqref{eq: defns of parameters}), and `flatten' the equations using \emph{Dirichlet-Neumann formalism}.
The \emph{Dirichlet-Neumann operator} $G^\prime$ for the upper fluid domain 
(given by $\{\eta(x,z)<y<\tfrac{1}{\beta_0}\}$ in dimensionless variables)
is defined as follows.
Fix $\Phi^\prime=\Phi^\prime(x,z),$ solve the \emph{linear} boundary-value problem
\begin{align*}
\phi^\prime_{xx} +\phi^\prime_{yy}+\phi^\prime_{zz} &=0,\hphantom{\Phi^\prime}\quad\qquad\eta<y<\tfrac{1}{\beta_0},
\\
\phi^\prime&=\Phi^\prime,\hphantom{0}\quad\qquad y=\eta, 
\\
\phi^\prime_{y}&=0,\hphantom{\Phi^\prime}\quad\qquad y=\tfrac{1}{\beta_0},
\end{align*}
and define
\begin{align}
G^\prime(\eta,\Phi^\prime)& =-(1+\eta_{x}^{2}+\eta_{z}^{2})^\frac{1}{2}
\phi^\prime_{n}\big|_{y=\eta}
 =-(
\phi^\prime_{y}-\eta_{x}\phi^\prime_{x}-\eta_{z}\phi^\prime_{z}
)\big|_{y=\eta}.
\label{eq: N(B)}
\end{align}
The \emph{Dirichlet-Neumann operator} $G$ for the lower fluid domain
$\{-\tfrac{1}{\beta_0} < y < \eta(x,z)\}$ is similarly defined as
\begin{align}
G(\eta,\Phi)& =(1+\eta_{x}^{2}+\eta_{z}^{2})^\frac{1}{2}\mu(|\grad(\phi+y)|)
\phi_{n}\big|_{y=\eta}\nonumber\\
& =\mu(|\grad(\phi+y)|)(
\phi_{y}-\eta_{x}\phi_{x}-\eta_{z}\phi_{z}
)\big|_{y=\eta},
\label{eq: N(A)}
\end{align}
where $\phi$ is the solution of the (in general \emph{nonlinear}) boundary-value problem
\begin{align}
\textnormal{div}(\mu(|\grad(\phi+y)|)\grad(\phi+y))&=0, \hphantom{\mu_1}\quad\qquad-\tfrac{1}{\beta_0}<y<\eta,
\label{eq: Intro Laplace for G-unten}
\\
\phi&=\Phi, \hphantom{\mu_1}\quad\qquad y=\eta,
\\
\mu(|\grad(\phi+y)|)(\phi_y+1)&=\mu(1), \quad\qquad y=-\tfrac{1}{\beta_0}.
\label{eq: Intro Bc for G-unten}
\end{align}
The nonlinearity of  \eqref{eq: Intro Laplace for G-unten}--\eqref{eq: Intro Bc for G-unten}
is inherited from that of the magnetisation law ${\mathbf M}=\mathbf{M}(\mathbf{H})$ (for a linear magnetisation law the value
of $\mu$ is constant and \eqref{eq: Intro Laplace for G-unten}, \eqref{eq: Intro Bc for G-unten} are replaced by respectively Laplace's equation and
a linear Neumann boundary condition). In Sections
\ref{ch: Mathematical description}\ref{sec: DN formalism} and \ref{sec: DiNOs} we show that $G^\prime$ and $G$ are
analytic functions of respectively $(\eta,\Phi^\prime)$ and $(\eta,\Phi)$ in suitable function spaces and use these operators to
recast the governing
equations in terms of the variables $\Phi^\prime=\phi^\prime|_{y=\eta}$ and $\Phi=\phi|_{y=\eta}$.
The 
mathematical problem is thus to solve a system of equations of the form
$$
{\mathcal G}(\gamma,(\eta,\Phi^\prime,\Phi))=0,
$$
where ${\mathcal G}: {\mathbb R} \times X_0 \rightarrow Y_0$
is given explicitly by the left-hand sides of equations \eqref{eq: NNBC1}--\eqref{eq: NNBC3}
and the function spaces $X_0$, $Y_0$ are specified in equation \eqref{eq: Defns of U0, B0}.
Observe that this problem exhibits rotational symmetry: it
is invariant under rotations through respectively $\pi$, $\frac{\pi}{2}$ and $\frac{\pi}{3}$ for
rolls, rectangles and hexagons,
and one may therefore replace $X_0$ and $Y_0$ by their subspaces of functions that are invariant under these rotations
(denoted by  $X_\mathrm{sym}$ and $Y_\mathrm{sym}$). 

In Section \ref{ch: CRT} we discuss the existence of small-amplitude solutions
to \eqref{eq: Nfinal} within the framework of \emph{analytic
Crandall-Rabinowitz local bifurcation theory }
(see Buffoni and Toland\linebreak
\cite[Chapter 8]{BuffoniToland}), using $\gamma$
as a bifurcation parameter.
According to that theory values $\gamma_0$ of
$\gamma$ at which non-trivial solutions bifurcate from
zero (clearly ${\mathcal G}(\gamma,0)=0$ for all values of $\gamma$)
necessarily have the property that the kernel of the linear operator
$L_{0}:=\mathrm{d}_{2}{\mathcal G}[\gamma_0, 0]: X_0 \rightarrow Y_0 $ is non-trivial.
We show that $\ker L_0$ is non-trivial if and only if 
$$\gamma_0= r(|{\bf k}|):=\left(\mu_1(\mu_1-1)^2\bigg(\mu_1|\mathbf{k}| \coth\frac{|\mathbf{k}|}{\beta_0} +S_1 |\mathbf{k}| \coth\frac{S_1|\mathbf{k}|}{\beta_0}\bigg)^{\!\!-1}-1\right)|\mathbf{k}|^2$$
for some $\mathbf{k}\in\mathscr{L}^{\star}\setminus \left\{ \mathbf{0}\right\},$  
where $\mu_1=\mu(1)$, $\dot{\mu}_1=\dot{\mu}(1)$ and
$S_1=\left(\mu_1/(\mu_1+\dot{\mu}_1)\right)^{1/2}$. Choosing\linebreak
$\beta_0<\mu_1(\mu_1-1)^2/(\mu_1+1)$ and $\omega$ so that
$$\gamma_0 = \left(\mu_1(\mu_1-1)^2\bigg(\mu_1\omega \coth \frac{\omega}{\beta_0}
+S_1 \omega \coth\frac{S_1\omega}{\beta_0}\bigg)^{\!\!-1}-1\right)\omega^2$$
is the unique maximum of the mapping $|\mathbf{k}| \mapsto r(|\mathbf{k}|)$, we find that
$$
\ker L_{0}
=
\langle{\left\{  \mathbf{v}\sin(\mathbf{k}\cdot\mathbf{x}),\mathbf{v}\cos(\mathbf{k}\cdot\mathbf{x}): \mathbf{k}\in\mathscr{L}^{\star} 
\ \mathrm{with} \ |\mathbf{k}|=\omega\right\}}\rangle,
$$
where ${\bf v} \in {\mathbb R}^3$ is given by equation \eqref{eq: bold v}
(see the discussion to Figure \ref{pic: disprel}); this value
$(\beta_0,\gamma_0)$
of $(\beta,\gamma)$ corresponds to the Rosensweig instability. The dimension 
of $\ker L_{0}$ is therefore determined by the number of vectors in $\mathscr{L}^{\star}$ with length $\omega$;
for rolls, rectangles and hexagons we find that $\dim \ker L_{0}$ is respectively $2$, $4$ and
$6$ (see Figure \ref{fig: length} and Sattinger \cite[Section 2]{Sattinger78} for a general discussion of this point).  Because the kernel of $L_{0}$ is  multidimensional, one can not use Crandall-Rabinowitz local bifurcation theory directly.
To overcome this problem we replace $X_0$ and $Y_0$ by $X_\mathrm{sym}$ and $Y_\mathrm{sym}$, thus restricting to
solutions that are invariant under rotations
through respectively $\pi$, $\frac{\pi}{2}$ and $\frac{\pi}{3}$ for rolls, rectangles and hexagons.
These restrictions ensure that $\dim \ker L_{0}=1$ 
with $\ker L_{0}=\langle{v_{0}}\rangle,$ 
where  $v_{0}=\mathbf{v}\,e_{1}(x,z)$ 
and
$$e_1(x,z) = \left\{ \begin{array}{ll} \cos\omega x & \qquad \mbox{(rolls)} \\
\cos\omega x+\cos\omega z & \qquad \mbox{(rectangles)} \\
\cos \omega x
+\cos   \frac{\omega}{2}
\left(x+\sqrt{3}z\right)
+\cos \frac{\omega}{2}
\left(x-\sqrt{3}z\right)  & \qquad \mbox{(hexagons).}
\end{array}\right.$$

Verifying the remaining conditions in the analytic Crandall-Rabinowitz local bifurcation theorem yields the following result.
\begin{theorem} \label{thm: main result 1}
The point $(\gamma_0, 0)$ is a local bifurcation point for \eqref{eq: Nfinal}, 
 that is 
 there exist $\varepsilon >0,$ open neighbourhoods $W_\mathrm{sym}$ of $(\gamma_0, 0)$ in $\mathbb{R}\times X_\mathrm{sym}$ and $V_\mathrm{sym} $ of $0$ in $X_\mathrm{sym}$
and  analytic functions 
$w:(-\varepsilon,\varepsilon)\rightarrow V_\mathrm{sym} $, $ \gamma: (-\varepsilon,\varepsilon)\rightarrow \mathbb{R}$ with  
$\gamma(0)=\gamma_0, w(0)=v_{0}$ 
such that ${\mathcal G}(\gamma(s),sw(s))=0$ for every $s\in (-\varepsilon,\varepsilon).$ 
Furthermore   
\begin{align*}
W_\mathrm{sym}\cap N
=
\left\{  (\gamma(s),sw(s)):\
0<|s|<\varepsilon\right\},
\end{align*}
where
$$
N=\left\{  (\gamma,v)\in \mathbb{R}\times(V_\mathrm{sym}  \setminus \left\{  0\right\}):{\mathcal G}(\gamma,v)=0\right\}.
$$
\end{theorem} 

In Section \ref{ch: trans, sub or super} we examine the bifurcating branches identified in
Theorem \ref{thm: main result 1}.

\begin{theorem} \label{thm: main result 2}
Branches of small-amplitude doubly periodic solutions to the ferrohydrostatic problem
bifurcate from the trivial solution at $\gamma=\gamma_0$. The bifurcation is
\begin{enumerate}\vspace{-\baselineskip}
\item transcritical in the case of hexagons,
\item super- or subcritical in the case of rolls and rectangles, depending upon the sign
of a coefficient $\gamma_2$ which is determined by $\mu$ and $\omega/\beta_0$.
\end{enumerate}
\end{theorem}

Explicit formulae for the coefficient $\gamma_2$ are given in some special cases
in Section \ref{ch: trans, sub or super} (such formulae are unwieldy, and it appears in
general more appropriate to calculate them numerically for a specific choice of $\mu$).
We note in particular that for constant $\mu$ (corresponding to a linear magnetisation law)
and very deep fluids, rolls bifurcate subcritically for $\mu<\mu_\mathrm{c}^1$ and
supercritically for $\mu>\mu_\mathrm{c}^1$, while rectangles bifurcate subcritically for
$\mu<\mu_\mathrm{c}^2$ and supercritically for $\mu>\mu_\mathrm{c}^2$, where
$$
\mu_\mathrm{c}^1=\frac{21}{11}+\frac{8}{11}\sqrt{5},
\qquad
\mu_\mathrm{c}^2=\frac{115+160\sqrt{2}+8\sqrt{184+11\sqrt{2}}}{141+128\sqrt{2}}.
$$
The same values were obtained by Silber and Knobloch \cite{SilberKnobloch88} in a discussion of normal forms for this bifurcation problem and confirmed by Lloyd, Gollwitzer, Rehberg 
and Richter \cite{LloydGollwitzerRehbergRichter15}  
as part of a wider numerical and experimental investigation.

Finally, we note that supercritical bifurcation of rolls is associated with (supercritical) bifurcation of
spatially localised patterns, whose existence has been established by dynamical-systems
arguments by Groves, Lloyd \& Stylianou \cite{GrovesLloydStylianou17}.

\section{Mathematical formulation}
\label{ch: Mathematical description}
\subsection{The physical problem}

\label{sec: The physical problem}

We consider two static immiscible perfect fluids in the regions
$$\Omega^\prime:=\{(x,y,z):\,\eta(x,z)<y<d\}, \qquad \Omega:=\{(x,y,z):\,-d<y<\eta(x,z)\}$$
separated by the free interface $\{y=\eta(x,z)\}$.
The upper, non-magnetisable fluid has unit relative permeability and density $\rho^\prime$,
while the lower is a ferrofluid with density $\rho$.
The relations between the magnetic fields $\mathbf{H}^\prime$, $\mathbf{H}$ and the induction fields $\mathbf{B}^\prime$, $\mathbf{B}$  
are given by the identities 
$$
\mu_{0}\mathbf{H}^\prime =\mathbf{B}^\prime,
\qquad\qquad
 \mu_{0}(\mathbf{H}+\mathbf{M}(\mathbf{H}))
=\mathbf{B}
$$
where $\mu_{0}$ is the vacuum permeability 
and $\mathbf{M}$ is the 
magnetic intensity of the ferrofluid. (Here, and in the remainder of this paper, equations for `primed' and `non-primed' quantities are supposed to hold in respectively $\Omega^\prime$ and $\Omega$.)  
We suppose that 
$$
\mathbf{M}(\mathbf{H})=m(|\mathbf{H}|)\frac{\mathbf{H}}{|\mathbf{H}|},
$$
where $m$ is a nonnegative function, so that in particular $\mathbf{M}$ and $\mathbf{H}$ are collinear. 
According to Maxwell's equations the magnetic and induction fields are respectively irrotational and solenoidal, 
and introducing 
magnetic potential functions $\phi^\prime, \phi$ with
\begin{align*}
-\grad\phi^\prime =\mathbf{H}^\prime,
\qquad\qquad
-\grad\phi =\mathbf{H},
\end{align*}
one finds that these potentials satisfy the equations 
\begin{align}
\divergence(\grad\phi^\prime)=0
,
\qquad\qquad
\divergence(\mu(|\grad\,\phi|)\grad\,\phi)=0,
\label{eq: Laplaciana}
\end{align}
in which  
$$
\mu(s)=1+\frac{m(s)}{s};
$$
we assume that $\mu:(0,\infty)\to\mathbb{R}$ is analytic and satisfies $\mu(1) + \dot{\mu}(1)>0$
(so that the linearised version of the equation for $\phi$ is elliptic). 
Observe that $\phi^\prime$ is harmonic while $\phi$ satisfies a nonlinear elliptic partial differential equation;
this nonlinearity is inherited from that of the magnetisation law ${\mathbf M}=\mathbf{M}(\mathbf{H})$
(for a linear magnetisation law the value of $\mu$ is constant and the equation for $\phi$ reduces
to Laplace's equation).
At the interface we have the magnetic conditions 
$$
\mathbf{H}^\prime\cdot \mathbf{t}_{1} =\mathbf{H}\cdot \mathbf{t}_{1},
\qquad\qquad
\mathbf{H}^\prime\cdot \mathbf{t}_{2} =\mathbf{H}\cdot \mathbf{t}_{2},
\qquad\qquad
\mathbf{B}^\prime\cdot \mathbf{n} =\mathbf{B}\cdot \mathbf{n},
$$
where 
\begin{align*}
\mathbf{t}_{1}
&=
\cfrac{(1,\eta_{{x}},0)^{\mathrm{T}}
}{
\sqrt{1+\eta^{2}_{x}}}
,
\qquad\qquad
\mathbf{t}_{2}
=
\cfrac{(0,\eta_{{z}},1)^{\mathrm{T}}
}{
\sqrt{1+\eta^{2}_{z}}}
,
\qquad\qquad
\mathbf{n}=\cfrac{(-\eta_{{x}},1,-\eta_{{z}})^{\mathrm{T}}
}{
\sqrt{1+\eta^{2}_{x}+\eta^{2}_{z}}}
\end{align*}
are the tangent and normal vectors to the interface; 
it follows that 
\begin{align}
\phi^\prime-\phi\big|_{y=\eta} =0
,
\qquad\qquad
\phi^\prime_{n}
-\mu(|\grad\,\phi|) \phi_{n}\big|_{y=\eta}  =0.
\label{eq: bc2}
\end{align}

The ferrohydrostatic Euler equations are given by 
\begin{align*}
-\grad{(p'_{0}+\rho^\prime gy)} =\mathbf{0}  
,
\qquad\qquad
\mu_{0}
\left({\mathbf{M}(\mathbf{H})\cdot
(\partial_{x},\partial_{z},\partial_{y})^{\mathrm T}}\right) {\mathbf{H}}
-\grad({p^{\star}+
\rho gy}) =\mathbf{0}
\end{align*}
(Rosensweig \cite[Section 5.1]{Rosensweig}),
in which   
$g$ is the acceleration due to gravity, 
$p'_{0}$ is the pressure in the upper fluid and
$p^{\star}$ is a composite of 
the magnetostrictive and  fluid-magnetic pressures.
The calculation 
\begin{align*}
\left(\mathbf{M}(\mathbf{H})\cdot
(\partial_{x},\partial_{y},\partial_{z})^{\mathrm T}\right)
{\mathbf{H}}
=
m(|{\mathbf{H}}|)\,
\grad\,|{\mathbf{H}}|
=
\grad
\left(
\int_{0}^{|\mathbf{H}|}
m(t)
\, \mathrm{d} t
\right)
\end{align*}
shows that these equations are equivalent to
\begin{align}
-(p'_{0}+\rho^\prime gy) =C'_{0} 
, 
\qquad\qquad 
\mu_{0}
\int_{0}^{|\mathbf{H}|}
m(t)
\, \mathrm{d} t
-(p^{\star}
+\rho gy) =C_{0}
,
 \label{eq: C1}
\end{align}
where  $C'_{0},C_{0}$ are constants. 
The ferrohydrostatic boundary condition 
is given by 
$$
p^{\star}
+\frac{\mu_{0}}{2}
({\mathbf{M}(\mathbf{H})\cdot \mathbf n})^{2}= p'_{0}+2\kappa\sigma
$$
for $y=\eta(x,z)$ (Rosensweig \cite[Section 5.2]{Rosensweig}), 
in which $\sigma>0$ is the coefficient of surface tension and 
\begin{align*}
2\kappa
& =-\cfrac{\eta_{{xx}}(1+\eta_{{z}}^{2})+\eta_{{zz}}(1+\eta_{{x}}^{2})
-2\eta_{{x}}\eta_{{z}}\eta_{{xz}}}{(1+\eta_{{x}}^{2}+\eta_{{z}}^{2})^{3/2}}
\end{align*}
is the mean curvature of the interface. 
Using \eqref{eq: C1}, we find that
\begin{align*}
\dfrac{\mu_{0}}{2}(\mathbf{M}(\mathbf{H})\cdot \mathbf n)^{2}
+{\mu_{0}}
\int_{0}^{|\mathbf{H}|}
m(t)
\, \mathrm{d} t
+C
+ ({\rho^\prime-\rho})g\eta
-2\sigma\kappa& = 0
,
\end{align*}
where $C=C'_{0}-C_{0},$ or equivalently  
\begin{align}
&
C
+\mu_{0}
\sqrt{1+\eta_{x}^{2}+\eta_{z}^{2}}
(\phi^\prime_{y} \phi^\prime_{n}
-\mu(|\grad\,\phi|)\phi_{y} \phi_{n}) 
\nonumber
\\
&\qquad\mbox{}
- ({\rho-\rho^\prime}) g\eta
-2\sigma\kappa 
-\mu_{0}
\left(\frac{1}{2}{|{\grad\,\phi^\prime}|^{2}-M(|{\grad\,\phi}|)}\right)
=0
\label{eq: bcc3}
\end{align}
for $y=\eta(x,z)$ with
\begin{align*}
M(s)
=
\int_{0}^{s}t\mu(t)\, \mathrm{d} t.
\end{align*} 
The requirement that a uniform magnetic field and flat interface solves the physical problem, that is
$(\eta_0,\phi^\prime_0,\phi_0)=(0,\mu(h) h y,h y)$
is a solution to \eqref{eq: Laplaciana}, \eqref{eq: bc2} and \eqref{eq: bcc3},  
leads us to choose\linebreak
$C=-\mu_{0}M(h)-\mu_{0}\mu(h)(\mu(h)-1)h^{2}/2$;
we write $(\tilde{\eta},\tilde{\phi}^\prime,\tilde{\phi}) = (\eta,\phi^\prime,\phi)+(\eta_0,\phi^\prime_0,\phi_0)$
(so that $(\tilde{\eta},\tilde{\phi}^\prime,\tilde{\phi})=(0,0,0)$ is the `trivial' solution)
and drop the tildes for notational simplicity.

The next step is to introduce dimensionless variables
\[
(\hat x,\hat y,\hat{z})=\frac{\mu_0h^2}{\sigma}(x,y,z),
\quad
\hat\eta:= \frac{\mu_0h^2}{\sigma}\eta,
\quad
\hat \phi^\prime:= \frac{\mu_0h}{\sigma}\phi^\prime,
\quad
\hat \phi:= \frac{\mu_0h}{\sigma}\phi
\]
and functions
\[
 \hat\mu(s):= \mu(hs),\qquad \hat M(s):=M(hs).
\]
We find that
\begin{align}
\phi^\prime_{xx}+\phi^\prime_{yy}+\phi^\prime_{zz} & =0, \qquad\qquad \eta(x,z)<y<\tfrac{1}{\beta_0},\label{eq: NLAPLACEO}\\
\divergence(\mu(|\grad(\phi +y)|)\grad(\phi +y)) & =0, \qquad\qquad -\tfrac{1}{\beta_0} < y < \eta(x,z),\label{eq: NLAPLACEU}
\end{align}
with boundary conditions
\begin{align}
\phi^\prime_{y}
& = 0, \qquad\qquad y=\tfrac{1}{\beta_0}, 
\label{eq: ENBC4}
\\
\mu(|\grad(\phi +y)|)( \phi_{y}+1)-\mu(1) 
& = 0, \qquad\qquad y=-\tfrac{1}{\beta_0},\label{eq: ENBC5}
\\
\phi^\prime-\phi
+\left({\mu(1)-1}\right)\eta& =0,\qquad\qquad
y=\eta(x,z),
\label{eq: ENBC1}\\
(\phi^\prime+\mu(1) y)_{n}
-\mu(|\grad(\phi +y)|)(\phi +y)_{n}
& =0,\qquad\qquad
y=\eta(x,z),
\label{eq: ENBC2}
\end{align}
and
\begin{align}
&
-M(1)-\mu(1)\left(\frac{1}{2}\mu(1)-1\right)- \gamma \eta
-2\sigma\kappa
\nonumber
\\
&\quad\mbox{}-\frac{1}{2}{|{\grad(\phi^\prime+\mu(1) y)}|^{2}+M(|{\grad(\phi+y)}|)}
\nonumber
\\
&\quad\mbox{}{
+
\sqrt{1+\eta_{x}^{2}+\eta_{z}^{2}}
({\phi^\prime_{y}+\mu(1)})
(\phi^\prime+\mu(1) y)_{n}
}
\nonumber
\\
&\quad\mbox{}{
-\sqrt{1+\eta_{x}^{2}+\eta_{z}^{2}}
\mu(|\grad(\phi +y)|)
({\phi_{y}+1})
(\phi+y)_{n}
}
 = 0, \qquad\qquad
y=\eta(x,z),
\label{eq: ENBC3}
\end{align}
where
\begin{equation}
 \alpha=\frac{(\rho-\rho^\prime)gd}{\mu_0h^2}, \qquad \beta=\frac{\sigma}{\mu_0h^2d}, \qquad \gamma=\alpha\beta
 \label{eq: defns of parameters}
 \end{equation}
 and the hats have been dropped for notational simplicity.
 
We seek periodic solutions to \eqref{eq: NLAPLACEO}--\eqref{eq: ENBC3} satisfying
$$
\eta(\mathbf{x} + \mathbf{l})=\eta(\mathbf{x})
,
\qquad
\phi^\prime(\mathbf{x} + \mathbf{l},y)=\phi^\prime(\mathbf{x},y)
,
\qquad
\phi(\mathbf{x} + \mathbf{l},y)=\phi(\mathbf{x},y)
$$
for every $\mathbf{l}\in \mathscr{L}$ (with a slight abuse of notation),
where $\mathbf{x}=(x,z)$ and $\mathscr{L}$ is the lattice given by
$$
\mathscr{L}=\left\{ m\mathbf{l}_{1}+n\mathbf{l}_{2}:\ m,n\in\mathbb{Z}\right\}
$$
with $\left|{\mathbf{l}_{1}}\right|=\left|{\mathbf{l}_{2}}\right|.$ 
Choose $\mathbf{k}_{1},\mathbf{k}_{2}$ with $\mathbf{k}_{i}\cdot\mathbf{l}_{j}=2\pi\delta_{ij}$ for $i,j=1,2$ and define the dual lattice $\mathscr{L}^{\star}$ to $\mathscr{L}$ by
$$
\mathscr{L}^{\star}=\left\{ {m\mathbf{k}_{1}+n\mathbf{k}_{2}:\ m,n\in\mathbb{Z}}\right\},
$$
 so that our periodic functions can be written as 
 $$
\eta(\mathbf{x})
=
\sum_{\mathbf{k}\in\mathscr{L}^{\star}}
\eta_{_{\mathbf{k}}}\mathrm{e}^{\mathrm{i}\mathbf{k}\cdot\mathbf{x}},
\qquad
\phi^\prime(\mathbf{x},y)
=
\sum_{\mathbf{k}\in\mathscr{L}^{\star}}
\phi^\prime_{_{\mathbf{k}}}(y)\mathrm{e}^{\mathrm{i}\mathbf{k}\cdot\mathbf{x}},
\qquad
\phi(\mathbf{x},y)
=
\sum_{\mathbf{k}\in\mathscr{L}^{\star}}
\phi_{_{\mathbf{k}}}(y)\mathrm{e}^{\mathrm{i}\mathbf{k}\cdot\mathbf{x}}
,
$$
where $\eta_{_{-\mathbf{k}}}=\bar{\eta}_{_{\mathbf{k}}},\phi^\prime_{_{-\mathbf{k}}}=\bar{\phi^\prime}_{_{\mathbf{k}}},\phi_{_{-\mathbf{k}}}=\bar{\phi}_{_{\mathbf{k}}}$. 
We are especially interested in three periodic patterns, namely rolls, rectangles and hexagons
(see Figure \ref{fig: lattices}).
\begin{enumerate}
\item\vspace{-\baselineskip}
For rolls we seek  functions that are independent of the $z$-direction   
and we choose\linebreak
$\mathbf{l}=(\frac{2\pi}{\omega},0),$  so that the dual lattice $\mathscr{L}^{\star}$ is 
generated by $\mathbf{k}=(\omega,0)$ 
and the periodic base cell  
is given by 
$$
\left\{ 
x:| x|<\frac{\pi}{\omega}
\right\}.
$$
Furthermore, the $z$-independent versions of  equations \eqref{eq: NLAPLACEO}--\eqref{eq: ENBC3} 
are invariant under the  reflection $x\mapsto -x$ (which corresponds to a rotation through ${\pi}$ in the $(x,z)$-plane).
\item
For rectangles we choose $\mathbf{l}_{1}=(\frac{2\pi}{\omega},0)$ and $\mathbf{l}_{2}=(0,\frac{2\pi}{\omega})$,  
so that the dual lattice $\mathscr{L}^{\star}$ is generated by $\mathbf{k}_{1}=({\omega},0)$  and $\mathbf{k}_{2}=(0,{\omega})$  
and the periodic base cell  
is given by 
$$
\left\{ 
(x,z):| x|,\:| z|<\frac{\pi}{\omega}\right\}.
$$
Furthermore, equations \eqref{eq: NLAPLACEO}--\eqref{eq: ENBC3} 
are invariant under rotations through  $\pi/2$ in the $(x,z)$-plane.
\item 
For hexagons we choose $\mathbf{l}_{1}=\frac{2\pi}{\omega}(1,-\frac{1}{\sqrt{3}})$ and 
$\mathbf{l}_{2}=\frac{2\pi}{\omega}(0,\frac{2}{\sqrt{3}}),$ 
so that we obtain an additional periodic direction 
$\mathbf{l}_{3}=\mathbf{l}_{1}+\mathbf{l}_{2}=\frac{2\pi}{\omega}(1,\frac{1}{\sqrt{3}})$.
The dual lattice 
$\mathscr{L}^{\star}$ is generated by 
$\mathbf{k}_{1}=(\omega,0)$  and $\mathbf{k}_{2}=\omega(\frac{1}{2},\frac{\sqrt{3}}{2})$
and the periodic base cell  
is given by
$$
\left\{ 
(x,z):|x|<\frac{2\pi}{\omega},\:
\left|{x-\sqrt{3}z}\right|<\frac{4\pi}{\omega}
\mbox{ and }
\left|{x+\sqrt{3}z}\right|<\frac{4\pi}{\omega}
\right\}.
$$ 
Furthermore, equations \eqref{eq: NLAPLACEO}--\eqref{eq: ENBC3} 
are invariant under rotations through  ${\pi}/{3}$ in the $(x,z)$-plane.

\end{enumerate}\vspace{-\baselineskip}
The mathematical problem is thus to solve \eqref{eq: NLAPLACEO}--\eqref{eq: ENBC3}  
for periodic functions $\eta, \phi^\prime$ and $\phi$ in the domains $\Gamma$, $\Omega_\mathrm{per}^\prime$ and $\Omega_\mathrm{per},$     
where 
$$
\Omega_\mathrm{per}^\prime:=
\left\{ 
(x,y,z):(x,z)\in\Gamma
\right\}\cap \Omega^\prime,
\qquad
\Omega_\mathrm{per}:=
\left\{ 
(x,y,z):(x,z)\in\Gamma
\right\}\cap \Omega
$$
and $\Gamma$ is the parallelogram defined by $\mathbf{l}_{1}$ and $\mathbf{l}_{2}$ (or by $\mathbf{l}$ in the case of rolls).

\begin{figure}[H] 
\centering\includegraphics[scale=0.9]{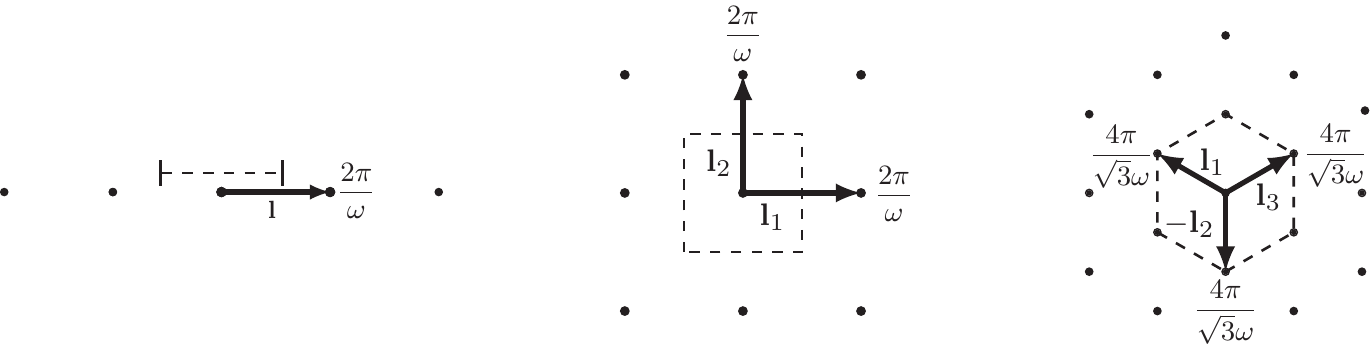}
\caption{\label{fig: lattices}
The  lattice $\mathscr{L}$ and periodic base cell  
for rolls (left), rectangles (centre) and hexagons (right). }
\end{figure}

\subsection{Dirichlet-Neumann formalism} \label{sec: DN formalism}

The \emph{Dirichlet-Neumann operator} $G^\prime$ for the upper fluid domain 
$\{\eta(x,z)<y<\tfrac{1}{\beta_0}\}$
is defined as follows.
Fix $\Phi^\prime=\Phi^\prime(x,z),$ solve the \emph{linear} boundary-value problem
\begin{align}
\phi^\prime_{xx} +\phi^\prime_{yy}+\phi^\prime_{zz} &=0,\hphantom{\Phi^\prime}\quad\qquad\eta<y<\tfrac{1}{\beta_0},
\label{eq: Laplace for G-oben}
\\
\phi^\prime&=\Phi^\prime,\hphantom{0}\quad\qquad y=\eta, 
\\
\phi^\prime_{y}&=0,\hphantom{\Phi^\prime}\quad\qquad y=\tfrac{1}{\beta_0} 
.
\label{eq: Bc for G-oben}
\end{align}
and define
\begin{align}
G^\prime(\eta,\Phi^\prime)& =-(1+\eta_{x}^{2}+\eta_{z}^{2})^\frac{1}{2}
\phi^\prime_{n}\big|_{y=\eta}
 =-(
\phi^\prime_{y}-\eta_{x}\phi^\prime_{x}-\eta_{z}\phi^\prime_{z}
)\big|_{y=\eta}.
\label{eq: N(B)}
\end{align}
The \emph{Dirichlet-Neumann operator} $G$ for the lower fluid domain
$\{-\tfrac{1}{\beta_0} < y < \eta(x,z)\}$ is similarly defined as
\begin{align}
G(\eta,\Phi)& =(1+\eta_{x}^{2}+\eta_{z}^{2})^\frac{1}{2}\mu(|\grad(\phi+y)|)
\phi_{n}\big|_{y=\eta}\nonumber\\
& =\mu(|\grad(\phi+y)|)(
\phi_{y}-\eta_{x}\phi_{x}-\eta_{z}\phi_{z}
)\big|_{y=\eta},
\label{eq: N(A)}
\end{align}
where $\phi$ is the solution of the (in general \emph{nonlinear}) boundary-value problem
\begin{align}
\textnormal{div}(\mu(|\grad(\phi+y)|)\grad(\phi+y))&=0, \hphantom{\mu_1}\quad\qquad-\tfrac{1}{\beta_0}<y<\eta,
\label{eq: Laplace for G-unten}
\\
\phi&=\Phi, \hphantom{\mu_1}\quad\qquad y=\eta,
\\
\mu(|\grad(\phi+y)|)(\phi_y+1)&=\mu(1), \quad\qquad y=-\tfrac{1}{\beta_0}.
\label{eq: Bc for G-unten}
\end{align}
It is also convenient to introduce auxiliary operators $H^\prime$ and $H$ given by
\begin{equation}
H^\prime(\eta,\Phi^\prime) = \phi^\prime_y|_{y=\eta}, \qquad H(\eta,\Phi) = \phi_y|_{y=\eta}, \label{eq: Defn of the Ws}
\end{equation}
where $\phi^\prime$ and $\phi$ are the solutions to the boundary-value problems \eqref{eq: Laplace for G-oben}--\eqref{eq: Bc for G-oben}
and \eqref{eq: Laplace for G-unten}--\eqref{eq: Bc for G-unten}.
Using this Dirichlet-Neumann formalism, we can recast the governing
equations \eqref{eq: NLAPLACEO}--\eqref{eq: ENBC3} in terms of the variables $\eta$,
$\Phi^\prime=\phi^\prime|_{y=\eta}$ and $\Phi = \phi|_{y=\eta}$ as 
\begin{align}
\Phi^\prime- \Phi +(\mu(1)-1) \eta& =0,\label{eq: NNBC1} \\
G^\prime(\eta,\Phi^\prime) + G(\eta,\Phi)
+(\mu^{\star} - \mu(1))& =0\label{eq: NNBC2}
\end{align}
and
\begin{align}
&-\gamma\eta+ 
\nabla\cdot\left({ 
\cfrac{\nabla \eta}{\sqrt{1+ |{\nabla{\eta}}|^{2}}}
}\right)
+\frac{1}{2}
\left({1+\left|{\nabla\eta}\right|^{2}}\right)
H^\prime(\eta,\Phi^\prime)^{2}
-\frac{1}{2}|{\nabla\Phi^\prime}|^{2}\Big|_{y=\eta}
\nonumber
\\
&\qquad\mbox{}
-\mu(1)G^\prime(\eta,\Phi^\prime) - G(\eta,\Phi)-(\mu^{\star} - \mu(1))
\nonumber
\\
&\qquad\mbox{}
+(M^{\star}-M(1))-\mu^{\star}H(\eta,\Phi)-H(\eta,\Phi) G(\eta,\Phi)
=0
,\label{eq: NNBC3}
\end{align}
in which 
$\nabla=(\partial_{x},\partial_{z})^{\mathrm{T}},$ $M(s)=\displaystyle\int_0^s t\mu(t)\,\mathrm{d}t$,
\begin{align*}
\mu^{\star}
&=
\mu
\!
\left(
\!
\left({
|\nabla\Phi|^{2}
+2(1-\nabla\eta\cdot\nabla\Phi)H(\eta,\Phi)
+(1+|\nabla\eta|^{2})H(\eta,\Phi)^{2}
+1
}
\right)^{{1}/{2}}
\right),
\\
M^{\star}
&=
M
\!
\left(
\!
\left({
|\nabla\Phi|^{2}
+2(1-\nabla\eta\cdot\nabla\Phi)H(\eta,\Phi)
+(1+|\nabla\eta|^{2})H(\eta,\Phi)^{2}
+1
}
\right)^{{1}/{2}}
\right).
\end{align*}

We study equations \eqref{eq: NNBC1}--\eqref{eq: NNBC3} in the standard Sobolev spaces 
$$
H_{\mathrm{per}}^{r}(\Gamma)
=
\left\{ 
\zeta
=
\sum_{\mathbf{k}\in\mathscr{L}^{\star}}
\zeta_{_{\mathbf{k}}}\mathrm{e}^{\mathrm{i}\mathbf{k}\cdot\mathbf{x}}: \zeta_{-\mathbf{k}}=\overline{\zeta_\mathbf{k}},
\,\left\lVert 
\zeta
\right\rVert_{r}
<
\infty
\right\}, \qquad
\left\lVert \zeta\right\rVert_{r}^2 
:= C(\Gamma)\!\!
\sum_{\mathbf{k}\in\mathscr{L}^{\star}}
\left( 1 + |\mathbf{k}|^2 \right)^{r} 
|\zeta_{_{\mathbf{k}}}|^2,
$$
where $\mathscr{L}^{\star}$ is the dual lattice to $\mathscr{L}$, and their subspaces 
\begin{align*}
\bar{H}_{\mathrm{per}}^{r}(\Gamma) 
&= 
\left\{
\zeta \in H^{r}_{\mathrm{\mathrm{per}}}(\Gamma) : 
\zeta_{_{\mathbf{0}}}=0
\right\}
\end{align*}  
consisting of functions with zero mean (the value of the normalisation constant $C(\Gamma)$ is $2\pi/\omega$ for rolls, $4(\pi/\omega)^{2}$ for rectangles and
$8/\sqrt{3}(\pi/\omega)^{2}$ for hexagons).
In Section \ref{ch: Mathematical description}\ref{sec: DiNOs} below
we establish the following theorem. (A function is `analytic at the origin' if is defined
and analytic in a neighbourhood of the origin; in particular it has a convergent
Maclaurin series.)

\begin{theorem}\label{thm: NEins}
Suppose that $s>5/2$. Formulae \eqref{eq: N(B)}, \eqref{eq: N(A)} and \eqref{eq: Defn of the Ws} define 
mappings $G^\prime$, $G: H^{s}_{\mathrm{per}}(\Gamma) \times H^{s-1/2}_{\mathrm{per}}(\Gamma) \rightarrow
 H^{s-3/2}_{\mathrm{per}}(\Gamma)$ and $H^\prime$, $H: H^{s}_{\mathrm{per}}(\Gamma) \times H^{s-1/2}_{\mathrm{per}}(\Gamma) \rightarrow
 H^{s-3/2}_{\mathrm{per}}(\Gamma)$ 
which are  analytic at the origin.
\end{theorem}
Define
\begin{equation}
X_0 :=
H^{s+1/2}_{\mathrm{per}}(\Gamma)
\times
\bar{H}^{s}_{\mathrm{per}}(\Gamma)
\times
H^{s}_{\mathrm{per}}(\Gamma),
\qquad
 Y_0
:=
H^{s}_{\mathrm{per}}(\Gamma)
\times
\bar{H}^{s-1}_{\mathrm{per}}(\Gamma)
\times
H^{s-3/2}_{\mathrm{per}}(\Gamma) 
\label{eq: Defns of U0, B0}
\end{equation}
for $s>5/2$. Using Theorem \ref{thm: NEins} and the fact that $H^r(\Gamma)$ is a Banach algebra for $r>1$, we find that
the left-hand sides of equations \eqref{eq: NNBC1}--\eqref{eq: NNBC3}  
define a function ${\mathcal G}:\mathbb{R}\times X_0 \to Y_0$ which is analytic at the origin. (A straightforward calculation shows that
$$
\int_\Gamma \left(G^\prime(\eta,\Phi^\prime) + G(\eta,\Phi)+(\mu^{\star} - \mu(1))\right)=0
$$
and explains the choice of functions with zero mean in the second component of $Y_0$; using functions
with zero mean in the second component of $X_0$ on the other hand ensures that the kernel of the linear operator
$\mathrm{d}_{2}{\mathcal G}[\gamma,0]$  does not contain any constant terms for any $\gamma\in\mathbb{R}$).
The 
mathematical problem is thus to solve\pagebreak
\begin{equation}\label{eq: Nfinal}
{\mathcal G}(\gamma,(\eta,\Phi^\prime,\Phi))=0,
\end{equation}
for $(\gamma,(\eta,\Phi^\prime,\Phi)) \in {\mathbb R} \times V_0$, where $V_0$ is a neighbourhood of the origin in $X_0$ and
${\mathcal G}(\gamma, 0)=0$ for all $\gamma \in \mathbb{R}.$  Observe that this problem exhibits rotational symmetry: it
is invariant under rotations through respectively $\pi$, $\frac{\pi}{2}$ and $\frac{\pi}{3}$
for rolls, rectangles and hexagons,
and one may therefore replace $X_0$ and $Y_0$ by their subspaces of functions that are invariant under these rotations
(denoted by  $X_\mathrm{sym}$ and $Y_\mathrm{sym}$).

 \subsection{Analyticity of the Dirichlet-Neumann operators}
 \label{sec: DiNOs}
 
We study the boundary-value problems \eqref{eq: Laplace for G-unten}--\eqref{eq: Bc for G-unten} 
and \eqref{eq: Laplace for G-oben}--\eqref{eq: Bc for G-oben} by transforming them into equivalent problems
in fixed domains (cf. Nicholls and Reitich \cite{NichollsReitich01a} and 
Twombly and Thomas \cite{TwomblyThomas83}). 
The change of variable
\begin{equation}
\tilde{y}
=\cfrac{y-\eta}{1+\beta_0\eta}, \qquad
u(x,\tilde{y},z) =\phi(x,y,z)\label{eq:TT below}
\end{equation}
transforms the variable domain $\Omega_\mathrm{per}$ into the fixed domain   
$\Sigma=\left\{{(x,y,z): (x,z)\in\Gamma, y \in (-\tfrac{1}{\beta_0},0)}\right\}$
and the boundary-value problem \eqref{eq: Laplace for G-unten}--\eqref{eq: Bc for G-unten} into
\begin{align}
\textnormal{div} \big(\mu^\dag(\grad(u+y)-(F_1(\eta,u),F_3(\eta,u),F_2(\eta,u))^\mathrm{T})\big)
&
=0,
&&-\tfrac{1}{\beta_0}<y<0,
\label{eq: Laplace for G-unten a. t.}
\\
u-\Phi
&
=0,
&&y=0
,
\label{eq: Bc for G-unten a. t.Stern}
\\
\mu^\dag\big(\grad(u+y).(0,1,0)^\mathrm{T}-F_3(\eta,u)\big)-\mu(1)
&
=0,
&&y=-\tfrac{1}{\beta_0},
\label{eq: Bc for G-unten a. t.}
\end{align}
where
$$
 F_{1}(\eta, u)  =-\eta u_{x} +(1+ \beta_0 y)\eta_{x} u_{y},\qquad
 F_{2}(\eta, u) =-\eta u_{z} +(1+ \beta_0 y)\eta_{z} u_{y},
$$
$$
 F_{3}(\eta, u) =\frac{\beta_0\eta u_{y}}{1+\beta_0\eta}
+(1+ \beta_0 y)(\eta_{x} u_{x}+ \eta_{z} u_{z})
-\frac{(1+\beta_0 y)^{2}}{1+\beta_0\eta}
(\eta^{2}_{x}+\eta^{2}_{z})u_{y}
$$
and
$$\mu^\dag = \mu\left(\left| \frac{1}{1+\beta_0\eta}\big(\grad u - (F_1(\eta,u),0,F_2(\eta,u))^\mathrm{T}\big)+(0,1,0)^\mathrm{T}\right|\right)$$
(we have again dropped the tildes for notational simplicity).

\begin{theorem}\label{thm: uanal}
Suppose that $s> 5/2.$
There exist open neighbourhoods $V$ and $U$ of the origin in respectively
$H_\mathrm{per}^s(\Gamma) \times H_\mathrm{per}^{s-1/2}(\Gamma)$ and
$H_\mathrm{per}^s(\Sigma)$ such that
the boundary-value problem \eqref{eq: Laplace for G-unten a. t.}--\eqref{eq: Bc for G-unten a. t.}
has a unique solution $u=u(\eta,\Phi)$ in $U$ for each $(\eta,\Phi) \in V$. Furthermore $u(\eta,\Phi)$ depends
analytically upon $\eta$ and $\Phi$.

\end{theorem}
\begin{proof}
Write the left-hand sides of equations \eqref{eq: Laplace for G-unten a. t.}--\eqref{eq: Bc for G-unten a. t.}
as
$${\mathcal H}(u,\eta,\Phi)=0,$$
and observe that ${\mathcal H}: H_\mathrm{per}^s(\Sigma) \times H_\mathrm{per}^s(\Gamma) \times
H_\mathrm{per}^{s-1/2}(\Gamma)$ is analytic at the origin with
$${\mathcal H}(0,0,0)=0.$$
Furthermore, the calculation
$$\mathrm{d}_1{\mathcal H}[0,0,0](u)=\begin{pmatrix} \mu_1 (u_{xx}+S_1^{-2} u_{yy}+u_{zz} ) \\
u|_{y=0} \\
\mu_1 S_1^{-2} u_y|_{y=-\frac{1}{\beta_0}}\end{pmatrix},$$
where
$\mu_1=\mu(1)$, $\dot{\mu}_1=\dot{\mu}(1)$ and
$S_1=\left(\mu_1/(\mu_1+\dot{\mu}_1)\right)^{1/2}$,
and standard existence and regularity theory for elliptic linear boundary-value problems show that
$\mathrm{d}_1{\mathcal H}[0,0,0]: H_\mathrm{per}^s(\Sigma) \rightarrow H_\mathrm{per}^{s-2} \times H_\mathrm{per}^{s-1/2} \times H_\mathrm{per}^{s-3/2}(\Gamma)$
is an isomorphism. The stated result now follows from the analytic implicit-function theorem.
\end{proof}

The transformation \eqref{eq:TT below} converts the formulae
$$G(\eta,\Phi)=\mu(|\grad(\phi+y)|)(\phi_{y}-\eta_{x}\phi_{x}-\eta_{z}\phi_{z})\big|_{y=\eta}, \qquad
H(\eta,\Phi)=\phi_y\big|_{y=\eta}$$
into
\begin{equation}
G(\eta,\Phi) = \mu^\dag (u_y-F_3(\eta,u))\big|_{y=0}, \qquad H(\eta,\Phi)=\left.\frac{u_y}{1+\beta_0\eta}\right|_{y=0},
\label{eq: Transformed G,W}
\end{equation}
where $u=u(\eta,\Phi)$ is the unique solution to \eqref{eq: Laplace for G-unten a. t.}--\eqref{eq: Bc for G-unten a. t.}.
\begin{theorem}\label{def: G unten}
Suppose that $s>5/2$. The formulae \eqref{eq: Transformed G,W} define analytic functions $G$,
$H: V \rightarrow H_\mathrm{per}^{s-3/2}(\Gamma)$.
\end{theorem}

To compute the Taylor-series representations of $u$ and $G$ we begin with the function\linebreak
$\nu: (H^{s-1}_{\mathrm{per}}(\Sigma))^{3}\to H^{s-1}_{\mathrm{per}}(\Sigma)$ defined by
$$\nu(T)=\mu(|T+(0,1,0)^\mathrm{T}|).$$
Observing that $\nu$ is analytic at the origin, we write its Taylor series as
\begin{align}\label{eq: nu as series}
\nu(T)
=
\sum_{j=0}^{\infty}
\nu^{j}(\{T\}^{(j)}),
\end{align}
where $\nu^{j}\in\mathcal{L}^{j}_{\mathrm{s}}((H^{s-1}_{\mathrm{per}}(\Sigma))^{3},H^{s-1}_{\mathrm{per}}(\Sigma))$ is given by
$$
\nu^{j}(T_1,\ldots,T_j) = \frac{1}{j!}\mathrm{d}^j\nu[0](T_1,\ldots,T_j)
$$
and may be computed explicitly from $\mu$ (note in particular that $\nu^0=\mu_1$).
The functions\linebreak
 $u^n \in\mathcal{L}^{n}_{\mathrm{s}}(H^{s}_{\mathrm{per}}(\Gamma)\times H^{s-1/2}_{\mathrm{per}}(\Gamma),H^{s}_{\mathrm{per}}(\Sigma))$  
(with $u^0=0$) in the corresponding series
\begin{equation}\label{eq: u as series}
u(\eta,\Phi)
=
\sum^{\infty}_{n=0}
u^{n}
\big(\{(\eta,\Phi)\}^{(n)}\big)
\end{equation}  
may be computed recursively by substituting the \emph{Ansatz} \eqref{eq: nu as series}, \eqref{eq: u as series}
into equations \eqref{eq: Laplace for G-unten a. t.}--\eqref{eq: Bc for G-unten a. t.}.
Consistently abbreviating $m^{n}\big(\{(\eta,\Phi)\}^{(n)}\big)$ to $m^{n}$ for notational simplicity, one finds after a lengthy but straightforward calculation that  
\begin{align*}
\divergence(L\,\grad\,u^{1})&=0, & \divergence(L\,\grad\,u^{n}-F^{n})& =0, \qquad\qquad -\tfrac{1}{\beta_0}<y<0, \\
u^{1}-\Phi&=0, & u^{n} & = 0, \qquad\qquad y=0, \\
u^{1}_{y}&=0, & (L\,\grad\,u^{n}-F^{n})\cdot(0,1,0)^{\mathrm{T}}&=0, \qquad\qquad y=-\tfrac{1}{\beta_0}
\end{align*}
for $n \geq 2$,  where 
 \begin{align*}
L
&=
\begin{pmatrix}
1&0&0 \\
0&S_1^{-2}&0 \\
0&0&1
\end{pmatrix},
\\
\mu_1F^{n}
&=
\mu_1
\!\left(\!
(F^{n}_{1},F^{n}_{3},F^{n}_{2})^{\mathrm{T}}
\!-
\frac{\dot{\mu}_1}{\mu_1} \sum_{j=1}^{n}\left(-\beta_0\eta\right)^{j}u^{n-j}_{y}(0,1,0)^{\mathrm{T}}
\!\!\right)\nonumber
\\
&\qquad\mbox{}-
\sum_{j=0}^{n}
\nu^{1}(T^{j})(\grad\, u^{n-j}-(F^{n-j}_{1},F^{n-j}_{3},F^{n-j}_{2})^{\mathrm{T}})
\nonumber
\\
&\qquad\mbox{}
-
R^{n}(0,1,0)^{\mathrm{T}}
-\sum_{j=0}^{n}
R^{j}(\grad\, u^{n-j}-(F^{n-j}_{1},F^{n-j}_{3},F^{n-j}_{2})^{\mathrm{T}})
\end{align*}
and\pagebreak
\begin{align*}
 F^{n}_{1}& =-\beta_0\eta u^{n-1}_{x} +(1+ \beta_0 y)\eta_{x} u^{n-1}_{y},\qquad
 F^{n}_{2}=-\beta_0\eta u^{n-1}_{z} +(1+ \beta_0 y)\eta_{z} u^{n-1}_{y},\\
 F^{n}_{3}& =\beta_0\eta \sum^{n-1}_{j=0}
 \left(-\beta_0\eta\right)^{j}u^{n-1-j}_{y}
+(1+\beta_0 y)(\eta_{x} u^{n-1}_{x}+ \eta_{z} u^{n-1}_{z})
\nonumber\\
&\qquad\mbox{}
-(1+\beta_0 y)^{2}
(\eta^{2}_{x}+\eta^{2}_{z})
\sum^{n-2}_{j=0}
 \left(-\beta_0\eta\right)^{j}u^{n-2-j}_{y},
\\
T^{n}
&=
\sum_{j=0}^{n}\left(-\beta_0 \eta\right)^{j}
\left(
\grad\, u^{{n-j}}-(F^{{n-j}}_{1},0,F^{{n-j}}_{2})^{\mathrm{T}}
\right)
,
\qquad
R^{n}
=\hspace{-4mm}
\sum_{
\substack{2\leq j\leq n, \\ 
\substack{
h_{1}+\ldots+h_{j}=n}}
}\hspace{-4mm}
\nu^{j}
\!(
T^{h_{1}},\dots,T^{h_{j}}
)
.
\end{align*}

The Taylor-series representations of $G$ and $H$ are thus given by  
$$G(\eta,\Phi)
=
\sum^{\infty}_{n=0}
G_{n}
\big(\{(\eta,\Phi)\}^{(n)}\big),
\qquad
H(\eta,\Phi)
=
\sum^{\infty}_{n=0}
H_{n}
\big(\{(\eta,\Phi)\}^{(n)}\big),
$$
where 
$$
G_n
=
\left.
\mu_1
I^n
+\sum_{j=0}^n \nu^1(T^j)I^{n-j}+\sum_{j=0}^n R^jI^{n-j}
\right|_{y=0} , \qquad
H_n
=
\sum_{j=0}^{n}
(-\beta_0\eta)^ju^{n-j}_y\Bigg|_{y=0} ,
$$
and 
\begin{align*}
I^{n}
&= 
\sum_{j=0}^{n}\left(-\beta_0 \eta
\right)^{j} u_{y}^{n-j}
+\sum_{j=0}^{n-2}\left(-\beta_0\eta
\right)^{j} 
(\eta^{2}_{x}+\eta^{2}_{z})
u_{y}^{n-j-2}
-(\eta_{x} u^{n-1}_{x}+ \eta_{z} u^{n-1}_{z}).
\end{align*} 
For later use we record the formulae
\begin{align*}
G_{1}
&=
\mu_1
u_{y}^{1}
\big|_{y=0},
\\
G_{2}
&=
\mu_1
\left(
\sum_{j=0}^{2}
\left(-\beta_0\eta
\right)^{j} 
u_{y}^{2-j}
-(\eta_{x} u^{1}_{x}+ \eta_{z} u^{1}_{z})
\right)
+
\dot{\mu}_1
(u_{y}^{1})^{2}
\Bigg|_{y=0},
\\
G_{3}
&=
\mu_1
\left(
\sum_{j=0}^{3}
\left(-\beta_0\eta
\right)^{j} u_{y}^{3-j}
+ 
(\eta^{2}_{x}+\eta^{2}_{z})
u_{y}^{1}
-
(\eta_{x} u^{2}_{x}+ \eta_{z} u^{2}_{z})
\right)
\\
&\qquad\mbox{}
+
\dot{\mu}_1
\left(
2\sum_{j=0}^{2}
\left(-\beta_0\eta
\right)^{j} 
u_{y}^{2-j}
-(\eta_{x} u^{1}_{x}+ \eta_{z} u^{1}_{z})
\right)\!\! u^{1}_{y}
+
\tfrac{1}{2}\ddot{\mu}_1
(u_{y}^{1})^{3}
+\tfrac{1}{2}\dot{\mu}_1\big((u_x^1)^2+(u_z^1)^2\big)u_y^1
\Bigg|_{y=0},
\end{align*}
where $\ddot{\mu}_1=\ddot{\mu}(1)$, and
$$
H_{1}=u^{1}_{y}\big|_{y=0},
\qquad
H_{2}=\sum_{j=0}^{2}\left(-\beta_0\eta\right)^{j}u^{2-j}_{y}\Bigg|_{y=0},
\qquad
H_{3}=\sum_{j=0}^{3}\left(-\beta_0\eta\right)^{j}u^{3-j}_{y}\Bigg|_{y=0}
$$
for the first few terms in these series.

The boundary-value problem \eqref{eq: Laplace for G-oben}--\eqref{eq: Bc for G-oben} is handled in a similar fashion. The change of variable
$$
\tilde{y}
=\frac{y-\eta}{1-\beta_0\eta}, \qquad
u^\prime(x,\tilde{y},z) =\phi^\prime(x,y,z)
$$
transforms the variable domain $\Omega_\mathrm{per}^\prime$ into the fixed domain   
$\Sigma^\prime=\left\{{(x,y,z):(x,z)\in\Gamma, y \in(0,\tfrac{1}{\beta_0})}\right\}$
and \eqref{eq: Laplace for G-oben}--\eqref{eq: Bc for G-oben} into
\begin{align}
\textnormal{div}\, \big(\grad u^\prime-(F^\prime_1(\eta,u^\prime),F^\prime_3(\eta,u^\prime),F^\prime_2(\eta,u^\prime))^\mathrm{T}\big)
&
=0,
\qquad&&
0<y<\tfrac{1}{\beta_0},
\label{eq: Laplace for G-oben a. t.}
\\
u^\prime-\Phi^\prime
&
=0,
&&y=0
,
\label{eq: Bc for G-oben a. t.Stern}
\\
u^\prime_y
&
=0,
&&y=\tfrac{1}{\beta_0}
,
\label{eq: Bc for G-oben a. t.}
\end{align}
where
$$
 F^\prime_{1}(\eta, u^\prime) =\eta u^\prime_{x} +(1-\beta_0 y)\eta_{x} u^\prime_{y},\qquad
 F^\prime_{2}(\eta, u^\prime)=\eta u^\prime_{z} +(1-\beta_0 y)\eta_{z} u^\prime_{y},
$$
$$
 F^\prime_{3}(\eta, u^\prime) =-\frac{\beta_0\eta u^\prime_{y}}{1-\beta_0\eta}
+(1-\beta_0 y)(\eta_{x} u^\prime_{x}+ \eta_{z} u^\prime_{z})
-\frac{(1-\beta_0y)^{2}}{1-\beta_0\eta}
(\eta^{2}_{x}+\eta^{2}_{z})u^\prime_{y}
$$
(we have again dropped the tildes for notational simplicity). The formulae
$$G^\prime(\eta,\Phi^\prime) =-(\phi^\prime_{y}-\eta_{x}\phi^\prime_{x}-\eta_{z}\phi^\prime_{z})\big|_{y=\eta}, \qquad
H^\prime(\eta,\Phi^\prime)=\phi^\prime_y\big|_{y=\eta}$$
are converted into
\begin{equation}
G^\prime(\eta,\Phi^\prime) = -u_y^\prime+F_3^\prime(\eta,u^\prime)\big|_{y=0}, \qquad
H^\prime(\eta,\Phi^\prime) = \left.\frac{u_y^\prime}{1-\beta_0\eta}\right|_{y=0}, \label{eq: Transformed G,W prime}
\end{equation}
where $u^\prime=u^\prime(\eta,\Phi^\prime)$ is the unique solution to \eqref{eq: Laplace for G-oben a. t.}--\eqref{eq: Bc for G-oben a. t.}.

\begin{theorem}\label{def: G oben}
Suppose that $s>5/2$.\vspace{-\baselineskip}
\begin{itemize}
\item[(i)]
There exist open neighbourhoods $V^\prime$ and $U^\prime$ of the origin in respectively
$H_\mathrm{per}^s(\Gamma) \times H_\mathrm{per}^{s-1/2}(\Gamma)$ and
$H_\mathrm{per}^s(\Sigma^\prime)$ such that
the boundary-value problem \eqref{eq: Laplace for G-oben a. t.}--\eqref{eq: Bc for G-oben a. t.}
has a unique solution\linebreak
 $u^\prime=u^\prime(\eta,\Phi^\prime)$ in $U^\prime$ for each $(\eta,\Phi^\prime) \in V^\prime$. Furthermore
$u^\prime(\eta,\Phi^\prime)$ depends analytically upon $\eta$ and $\Phi^\prime$.
\item[(ii)]
The formulae \eqref{eq: Transformed G,W prime}
define analytic functions $G^\prime$, $H^\prime: H^{s}_{\mathrm{per}}(\Gamma)\times H^{s-1/2}_{\mathrm{per}}(\Gamma) \rightarrow
 H^{s-3/2}_{\mathrm{per}}(\Gamma)$.
\end{itemize}
\end{theorem}

The functions $u^{\prime n}
\in\mathcal{L}^{n}_{\mathrm{s}}(H^{s}_{\mathrm{per}}(\Gamma)\times H^{s-1/2}_{\mathrm{per}}(\Gamma),H^{s}_{\mathrm{per}}(\Sigma^{\prime}))
$ 
(with $u^{\prime 0}=0$) in the Taylor series
$$
u'(\eta,\Phi^{\prime})
=
\sum^{\infty}_{n=1}
u^{\prime n}\big(\{(\eta,\Phi^{\prime})\}^{(n)}\big) 
$$
are computed recursively by substituting this \emph{Ansatz} into equations \eqref{eq: Laplace for G-oben a. t.}--\eqref{eq: Bc for G-oben a. t.};
one finds that
\begin{align*}
\divergence(\grad\,u^{\prime 1})&=0, & \divergence(\grad\,u^{\prime n}-(F^{\prime n}_{1},F^{\prime n}_{3},F^{\prime n}_{2})^{\mathrm{T}})& =0, \qquad\qquad 0<y<\tfrac{1}{\beta_0}, \\
u^{\prime 1}-\Phi^\prime&=0, & u^{\prime n} & = 0, \qquad\qquad y=0, \\
u^{\prime 1}_{y}&=0, & u^{\prime n}_{y}&=0, \qquad\qquad y=\tfrac{1}{\beta_0}
\end{align*}
for $n\geq 2$, where
$$
 F^{\prime n}_{1}=\eta u^{\prime n-1}_{x} +(1-\beta_0 y)\eta_{x} u^{\prime n-1}_{y},\qquad
 F^{\prime n}_{2} =\eta u^{\prime n-1}_{z} +(1-\beta_0 y)\eta_{z} u^{\prime n-1}_{y},
$$
\begin{align*}
F^{\prime n}_{3} =&-\beta_0\eta \sum^{n-1}_{j=0}
\left(\beta_0\eta\right)^{j}u^{\prime n-1-j}_{y}
+(1-\beta_0 y)(\eta_{x} u^{\prime n-1}_{x}+ \eta_{z} u^{\prime n-1}_{z}) \\
& \qquad\mbox{}
-(1-\beta_0y)^2(\eta^{2}_{x}+\eta^{2}_{z})\sum^{n-2}_{j=0}(\beta_0\eta)^ju^{\prime n-2-j}_{y}.
\end{align*}
The Taylor-series representations of $G^\prime$ and $H^\prime$ are given by 
$$
G^\prime(\eta,\Phi^{\prime})
=
\sum_{n=1}^{\infty}
G^\prime_{n}
\big(\{(\eta,\Phi^{\prime})\}^{(n)}\big),
\qquad
H^\prime(\eta,\Phi^{\prime})
=
\sum_{n=1}^{\infty}
H^\prime_{n}
\big(\{(\eta,\Phi^{\prime})\}^{(n)}\big)
$$
with 
\begin{align*}
G^\prime_{n}
&= 
\left.-
\sum_{j=0}^{n-1}\left(
\beta_0\eta
\right)^{j} u_{y}^{\prime n-j}
-\sum_{j=0}^{n-3}\left(
\beta_0\eta
\right)^{j} 
(\eta^{2}_{x}+\eta^{2}_{z})
u_{y}^{\prime n-j-2}
+(\eta_{x} u^{\prime n-1}_{x}+ \eta_{z} u^{\prime n-1}_{z})
\right|_{y=0}, \\
H^\prime_{n}
&= 
\sum_{j=0}^{n-1}\left(
\beta_0\eta
\right)^{j} u_{y}^{\prime n-j}
\Bigg|_{y=0},
\end{align*}
and in particular we find that
$$
G^\prime_{1}= -u_{y}^{\prime1}\big|_{y=0},
\qquad
G^\prime_{2}= -\left.\sum_{j=0}^{1}\left(\beta_0\eta\right)^{j} u_{y}^{\prime2-j}
+(\eta_{x} u^{\prime1}_{x}+ \eta_{z} u^{\prime1}_{z})\right|_{y=0},
$$
$$
G^\prime_{3}= -\left.\sum_{j=0}^{2}\left(\beta_0\eta\right)^{j} u_{y}^{\prime3-j}
-(\eta^{2}_{x}+\eta^{2}_{z})u_{y}^{\prime1}+(\eta_{x} u^{\prime2}_{x}+ \eta_{z} u^{\prime2}_{z})\right|_{y=0} .
$$
and
$$
H^\prime_{1}=u_{y}^{\prime1}\big|_{y=0},
\qquad
H^\prime_{2}= \left.\sum_{j=0}^{1}\left(\beta_0\eta\right)^{j} u_{y}^{\prime2-j}\right|_{y=0},
\qquad
H^\prime_{3} = \left.\sum_{j=0}^{2}\left(\beta_0\eta\right)^{j} u_{y}^{\prime3-j}\right|_{y=0}.
$$

\section{Existence theory}\label{ch: CRT}

Next we introduce the Crandall-Rabinowitz theorem (cf. Buffoni and Toland \cite[Theorem 8.3.1]{BuffoniToland}), 
an application of which yields  a local bifurcation point of the equation
\begin{equation}
{\mathcal G}(\gamma,(\eta,\Phi^\prime,\Phi))=0, \label{eq: Nfinal again}
\end{equation}
where the components of ${\mathcal G}: {\mathbb R} \times V_0 \rightarrow Y_0$ are given by
the left-hand sides of \eqref{eq: NNBC1}--\eqref{eq: NNBC3}. 

\begin{theorem}[Crandall-Rabinowitz theorem]\label{thm: CRT}\label{thm: CRT}
Let 
$X$ and $Y$ be Banach spaces, 
$V$ be an open neighbourhood of the origin in $X$   
and 
${\mathcal F}:\mathbb{R}\times V\rightarrow Y$ be an analytic function 
with ${\mathcal F}(\lambda, v) =0$ for all $\lambda \in \mathbb{R}.$ 
Suppose also that 
\hfill
\vspace{-\baselineskip}
\begin{enumerate}
\item
 $L:=\mathrm{d}_{2}{\mathcal F}[\lambda_{0}, 0]: X \rightarrow Y$ is a Fredholm operator of index zero, 
\item
$\ker L=\langle   v_{0} \rangle$ for some $v_{0} \in X,$  
\item
the transversality condition
$P(\mathrm{d}_{1}\mathrm{d}_{2}{\mathcal F}[\lambda_{{{0}}},0](1,v_{0}))\neq 0$ 
holds, 
where 
$P: Y\rightarrow Y$ is a projection 
with $\im L =\ker P.$
\hfill
\vspace{-\baselineskip}
\end{enumerate}
The point $(\lambda_{{{0}}},0)$ is a local bifurcation point, 
that is
there exist $\varepsilon >0,$ 
an open neighbourhood $W$ of $(\lambda_{0},0)$ in $\mathbb{R} \times X$
and  analytic functions 
$w:(-\varepsilon,\varepsilon)\rightarrow V,\ \lambda: (-\varepsilon,\varepsilon)\rightarrow \mathbb{R}$ with  
$\lambda(0)=\lambda_{0}, w(0)=v_{0}$ 
such that ${\mathcal F}(\lambda(s),sw(s))=0$ for every $s\in (-\varepsilon,\varepsilon).$ 
Furthermore   
$$
W\cap N=\left\{ (\lambda(s),sw(s)):0<|s|<\varepsilon\right\},
$$
where
$$
N=\left\{(\lambda,v)\in \mathbb{R}\times(V \setminus \left\{ 0 \right\}):{\mathcal F}(\lambda,v)=0 \right\}.
$$
\end{theorem}

The first step is to determine the maximal positive  value $\gamma_0$ 
of the parameter $\gamma$ 
for which the kernel of the linear operator 
$L_{0}:=\mathrm{d}_{2}{\mathcal G}[\gamma_0,0]:X_0 \to Y_0$, which is
given by the explicit formula    
\begin{equation}\label{eq: Linear Operator}
L_{0}\begin{pmatrix}\eta \\ \Phi^\prime \\ \Phi \end{pmatrix}
 =\begin{pmatrix}
\Phi^\prime-\Phi
+(\mu_1-1)\eta
\\
G_1^\prime(\eta,\Phi^\prime) + \mu G_1(\eta,\Phi)+\dot{\mu}_1H_1(\eta,\Phi)
\\
\eta_{xx}+\eta_{zz}-\gamma_0\eta-(\mu_1G_1^\prime(\eta,\Phi^\prime)+G_1(\eta,\Phi)+\dot{\mu}_1H_1(\eta,\Phi))
\end{pmatrix}
\end{equation}
with
\begin{align*}
G_1^\prime(\eta,\Phi^\prime)
=
\sum_{\mathbf{k}\in\mathscr{L}^{\star}}
|\mathbf{k}|\tanh  \frac{|\mathbf{k}|}{\beta_0} \Phi^\prime_{_{\mathbf{k}}}\mathrm{e}^{\mathrm{i}\mathbf{k}\cdot\mathbf{x}}
,
\qquad
G_1(\eta,\Phi)
=
\mu_1\sum_{\mathbf{k}\in\mathscr{L}^{\star}}
S_1|\mathbf{k}|\tanh \frac{S_1|\mathbf{k}|}{\beta_0} \Phi_{_{\mathbf{k}}}\mathrm{e}^{\mathrm{i}\mathbf{k}\cdot\mathbf{x}}
\end{align*}
and $G_1(\eta,\Phi)=\mu_1H_1(\eta,\Phi)$,
is non-trivial.  Writing  $v \in X_0$ 
as 
\begin{equation}
v(\mathbf{x})=\sum_{\mathbf{k}\in\mathscr{L}^{\star}}
\mathbf{v}_{_{\mathbf{k}}}\mathrm{e}^{\mathrm{i}\mathbf{k}\cdot\mathbf{x}} \label{eq: formula for v}
\end{equation}
with   
$\mathbf{v}_{_{\mathbf{k}}}=(\eta_{_{\mathbf{k}}},\Phi^\prime_{_\mathbf{k}},\Phi_{_\mathbf{k}})^{\mathrm{T}}$
and ${\bf v}_{_{-\mathbf{k}}}=\bar{\bf v}_{_{\mathbf{k}}}$, we find that    
\begin{align}\label{eq: linear Operator series}
L_{0}v
&=
\sum_{\mathbf{k}\in\mathscr{L}^{\star}}
L_{0}(|\mathbf{k}|)
\mathbf{v}_{_{\mathbf{k}}}\mathrm{e}^{\mathrm{i}\mathbf{k}\cdot\mathbf{x}},
\end{align}
where 
\begin{align*}
L_{0}(|\mathbf{k}|)
=
\begin{pmatrix}
\mu_1-1&1&-1
\\
0&|\mathbf{k}|\tanh \dfrac{ |\mathbf{k}|}{\beta_0}&\mu_1S_1^{-1}  |\mathbf{k}| \tanh  \dfrac{S_1|\mathbf{k}|}{\beta_0}
\\
-|\mathbf{k}|^{2}-\gamma_0
&
-\mu_1|\mathbf{k}|\tanh \dfrac{ |\mathbf{k}|}{\beta_0}
&
-\mu_1S_1^{-1}  |\mathbf{k}| \tanh  \dfrac{S_1|\mathbf{k}|}{\beta_0}
\end{pmatrix}
\end{align*}
for $\mathbf{k} \neq {\bf 0}$ and  
\begin{align*}
L_{0}(0)
&=
\begin{pmatrix}
\mu_1-1&1&-1
\\
0&0&0
\\
-\gamma_0
&0&0
\end{pmatrix}
\end{align*}
(where we have identified the subspace 
$\{ 
(\eta_{_{\mathbf{0}}},\Phi^\prime_{_{\mathbf{0}}},\Phi_{_{\mathbf{0}}})^{\mathrm{T}}
:\Phi^\prime_{_{\mathbf{0}}}=0\}$ of $\mathbb{R}^{3}$ with $\mathbb{R}^{2}$).

From this observation it follows that 
$\ker L_{0}$ is non-trivial if 
\begin{align*}
 \det L_{0}({{|\mathbf{k}|}})
&=\mu_1(\mu_1-1)^2S_1^{-1}|\mathbf{k}|^2 \tanh \frac{|\mathbf{k}|}{\beta_0}\tanh \frac{S_1|\mathbf{k}|}{\beta_0} \\
& \qquad\mbox{}-(|\mathbf{k}|^{2}+\gamma_0)\bigg(\mu_1|\mathbf{k}| \tanh\frac{|\mathbf{k}|}{\beta_0} +S_1 |\mathbf{k}| \tanh\frac{S_1|\mathbf{k}|}{\beta_0}\bigg) =0,
\end{align*}
that is
$$\gamma_0=r(|{\bf k}|):= \left(\mu_1(\mu_1-1)^2\bigg(\mu_1|\mathbf{k}| \coth\frac{|\mathbf{k}|}{\beta_0} +S_1 |\mathbf{k}| \coth\frac{S_1|\mathbf{k}|}{\beta_0}\bigg)^{\!\!-1}-1\right)|\mathbf{k}|^2$$
for some $\mathbf{k}\in\mathscr{L}^{\star}\setminus \left\{ \mathbf{0} \right\}$. The function $|\mathbf{k}| \mapsto r(|{\bf k}|)$, which satisfies $r(0)=0$ and $r(|\mathbf{k}|) \rightarrow-\infty$\linebreak
as $|\mathbf{k}| \rightarrow \infty$, takes only negative values for
$\beta_0>\mu_1(\mu_1-1)^2/(\mu_1+1)$, while for\linebreak
$\beta_0<\mu_1(\mu_1-1)^2/(\mu_1+1)$ it has a unique maximum $\omega$ with
$r(\omega)>0$ (see  Figure \ref{pic: disprel}); we choose $\gamma_0=r(\omega)$ and
note the relationships
$$\beta_0=\frac{\mu_1(\mu_1-1)^2}{2\tilde{\omega}}\left( \frac{h(\tilde{\omega})-\tilde{\omega}\dot{h}(\tilde{\omega})}{h(\tilde{\omega})^2}\right),
\qquad
\gamma_0=\left(\frac{\mu_1(\mu_1-1)^2}{\omega h(\tilde{\omega})}-1\right)\omega^2,$$
where $\tilde{\omega}=\omega/\beta_0$ and $h(\tilde{\omega})=\mu_1\coth \tilde{\omega}+S_1\coth S_1\tilde{\omega}$.

\begin{figure}[H]
\centering
\includegraphics[scale=0.62]{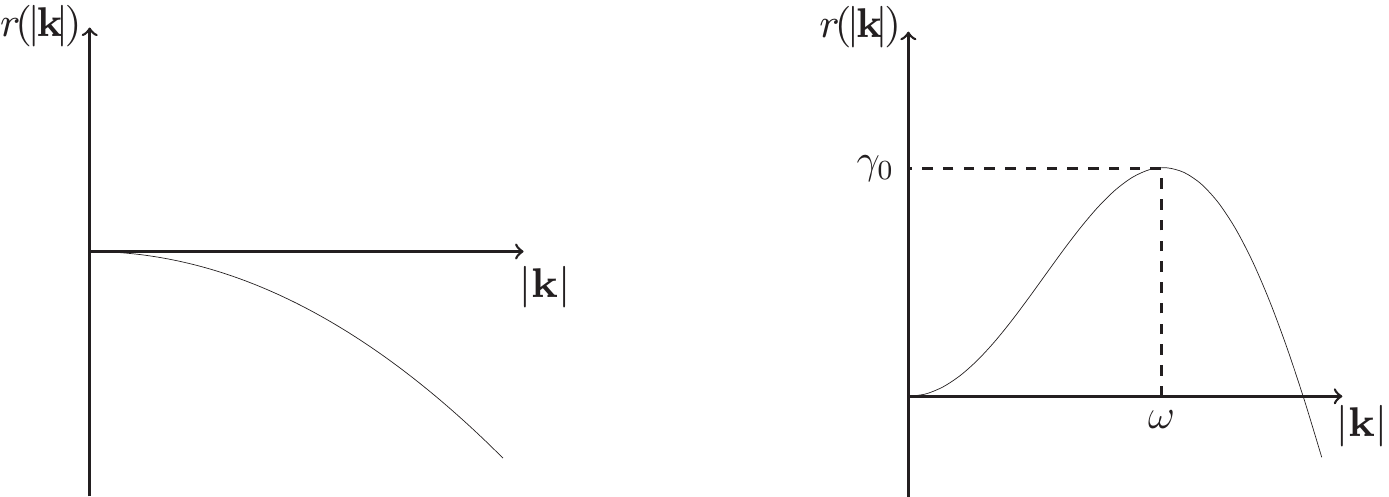}
\caption{\label{pic: disprel}
The graph of the function $|\mathbf{k}| \mapsto r(|\mathbf{k}|)$ for $\beta_0>\mu_1(\mu_1\!-\!1)^2/(\mu_1\!+\!1)$ (left) and $\beta_0<\mu_1(\mu_1\!-\!1)^2/(\mu_1\!+\!1)$ (right). 
} 
\end{figure}

Noting that
$\ker L_{0}(\omega) =\langle  \mathbf{v} \rangle,$ 
where  
\begin{equation}
\mathbf{v}=\left(\begin{array}{c} 
\cfrac{1}{\mu_1-1}\left(\mu_1S_1^{-1}\tanh S_1\tilde{\omega}\coth \tilde{\omega}+1\right) \\
-\mu_1S_1^{-1}\tanh S_1\tilde{\omega}\coth \tilde{\omega} \\
1\\
\end{array}
\right),
\label{eq: bold v}
\end{equation}
we find that  
$$
\ker L_{0}
=
\langle  \left\{ \mathbf{v}\sin(\mathbf{k}\cdot\mathbf{x}),\mathbf{v}\cos(\mathbf{k}\cdot\mathbf{x}): \mathbf{k}\in\mathscr{L}^{\star} 
\ \mathrm{with} \ |\mathbf{k}|=\omega \right\} \rangle
$$
(see Figure \ref{fig: length}).
\begin{enumerate}\vspace{-\baselineskip}
\item
For rolls the dual lattice $\mathscr{L}^{\star}$ is generated by $\mathbf{k}=(\omega,0),$ so that $|\mathbf{k}|,|-\mathbf{k}|=\omega$ 
and hence $\dim\ker L_{0}=2$.
\item
For rectangles the dual lattice $\mathscr{L}^{\star}$ is generated by $\mathbf{k}_{1}=(\omega,0)$ and $\mathbf{k}_{2}=(0,\omega),$ 
so that $|\mathbf{k}_{1}|,|-\mathbf{k}_{1}|,|\mathbf{k}_{2}|,|-\mathbf{k}_{2}|=\omega$ 
and hence $\dim\ker L_{0}=4$.
\item
For hexagons the dual lattice $\mathscr{L}^{\star}$ 
is generated by $\mathbf{k}_{1}=(\omega,0)$  and $\mathbf{k}_{2}=\omega(\frac{1}{2},\frac{\sqrt{3}}{2})$,
so that $|\mathbf{k}_{1}|,|-\mathbf{k}_{1}|,|\mathbf{k}_{2}|,|-\mathbf{k}_{2}|, |\mathbf{k}_{3}|,|-\mathbf{k}_{3}|=\omega,$ 
where $\mathbf{k}_{3}=\mathbf{k}_{2}-\mathbf{k}_{1},$ 
and hence $\dim \ker L_{0}=6$.
\end{enumerate}
\begin{figure}[H] 
\centering
\includegraphics[scale=0.9]{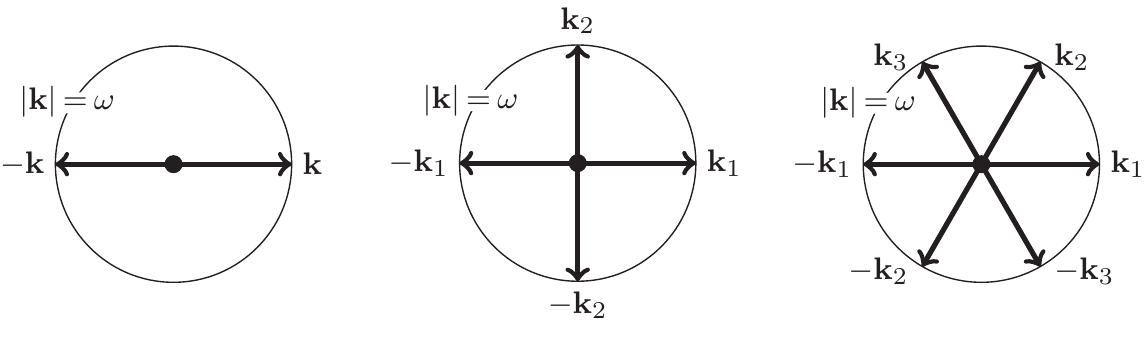}
\caption{\label{fig: length}
The vectors generated by $\mathscr{L}^{\star}$ with length $\omega$ in the case of  
rolls (left), rectangles (centre) and hexagons (right).}
\end{figure}

Define the projection $P_{\omega}:X_0\to X_0, Y_0\to Y_0$ 
by 
\begin{align*}
P_{\omega}
v(\mathbf{x})
&=
\sum_{\substack{\mathbf{k} \in\mathscr{L}^{\star}, \\ |\mathbf{k}|=\omega }}
\mathbf{v}_{_{\mathbf{k}}}\mathrm{e}^{\mathrm{i}\mathbf{k}\cdot\mathbf{x}},
\end{align*}
where $v$ is given by formula \eqref{eq: formula for v},
so that
$$
X_0=X_\omega \oplus X_{\mathrm{r}},
\qquad
Y_0=Y_\omega \oplus Y_{\mathrm{r}}
$$
with
\begin{align*}
 X_\omega=P_{\omega}[X_0],
\qquad
 X_{\mathrm{r}}=(I-P_{\omega})[X_0],
\qquad
 Y_\omega=P_{\omega}[Y_0],
\qquad
 Y_{\mathrm{r}}=(I-P_{\omega})[Y_0].
\end{align*}
Using \eqref{eq: linear Operator series}, 
one finds that 
$\im L_{0}\big|_{ X_\omega}\subseteq Y_\omega$ 
and 
$\im L_{0}\big|_{ X_{\mathrm{r}}}\subseteq Y_{\mathrm{r}}$  
and we prove 
that $L_{0}:X_0 \to Y_0$ is a Fredholm operator of index zero in two  steps.
\begin{lemma}\label{lem: Lo}
The mapping $L_{0}$ 
is an isomorphism 
$ X_{\mathrm{r}}\to Y_{\mathrm{r}}.$
\end{lemma}
\begin{proof}
The mapping $L_{0}\big|_{ X_{\mathrm{r}}}$ is formally invertible on ${ Y_{\mathrm{r}}}$ with
\begin{align}
L_{0}^{-1}(\chi,{{\Psi^\prime}},{{\Psi}})^\mathrm{T}
&=
\sum_{\substack{k\in\mathscr{L}^{\star} ,\\|\mathbf{k}|\neq \omega}}
L_{0}(|\mathbf{k}|)^{-1}
(\chi_{_{\mathbf{k}}},{{\Psi^\prime_{_{\mathbf{k}}}}},{{\Psi_{_{\mathbf{k}}}}})^{\mathrm{T}}\mathrm{e}^{\mathrm{i}\mathbf{k}\cdot\mathbf{x}},
\label{eq: SieheStern}
\end{align}
where 
\begin{align*}
L_{0}(|\mathbf{k}|)^{-1}
&=
\frac{
\mu_{1}S_{1}^{-1}|\mathbf{k}|\tanh \dfrac{S_{1}|\mathbf{k}|}{\beta_{0}}
}{
\det L_{0}(|\mathbf{k}|)
}
\begin{pmatrix}
(\mu_{1}-1) |\mathbf{k}|\tanh  \dfrac{|\mathbf{k}|}{\beta_{0} }
&
1
& 
1
\\[1em]
-(\gamma_{0}+ |\mathbf{k}|^{2})
&
-(\mu_{1}-1)
&
- (\mu_{1}-1)
\\[1em]
0
&
0
&
0
\end{pmatrix}
\\
\\ &\qquad\mbox{}+
\frac{
|\mathbf{k}|\tanh  \dfrac{|\mathbf{k}|}{\beta_{0} }
}{
\det L_{0}(|\mathbf{k}|)
}
\begin{pmatrix}
0
&
\mu_{1} 
&
1
\\[1em]
0
&
0
&
0
\\[1em]
\gamma_{0} +|\mathbf{k}|^{2}
&
\mu_{1}(\mu_{1}-1)
&
\mu_{1}-1
\end{pmatrix}
-\frac{\gamma_{0}+|\mathbf{k}|^{2}}{\det L_{0}(|\mathbf{k}|)}
\begin{pmatrix}
&0&0&0& \\[1em]
& 0 & 1 & 0 & \\[1em]
& 0 & 1 & 0 &
\end{pmatrix}
\end{align*}
for $\mathbf{k}\in\mathscr{L}^{\star}$ with $|\mathbf{k}|>\omega$ 
and 
\begin{align*}
L_{0}(0)^{-1}
\begin{pmatrix} \chi_{_{\mathbf{0}}} \\ 0 \\ \Psi_{_{\mathbf{0}}} \end{pmatrix}
=
\begin{pmatrix}
-\gamma_{0}^{-1}
\Psi_{_{\mathbf{0}}}
\\
0
\\
-\chi_{_{\mathbf{0}}}
-(\mu_{1}-1)\gamma_{0}^{-1}
\Psi_{_{\mathbf{0}}}
\end{pmatrix}
\end{align*}

Denoting the right-hand side of equation \eqref{eq: SieheStern} by $(\eta,\Phi^\prime,\Phi)^\mathrm{T}$
and using the estimates
\begin{align*}
{S_{1}^{-1} |\mathbf{k}|\tanh\dfrac{S_{1} |\mathbf{k}|}{\beta_{0}}}
\lesssim
{ |\mathbf{k}|\tanh\dfrac{|\mathbf{k}|}{\beta_{0}}}
,\qquad
\frac{|\mathbf{k}|\tanh\dfrac{|\mathbf{k}|}{\beta_{0}}}{|\det L_{0}(|\mathbf{k}|)|}
\lesssim
|\mathbf{k}|^{-2}
\end{align*}
for $|\mathbf{k}|> \omega$,
one finds that 
\begin{align*}
\left\lVert  \eta \right\rVert_{ s+{1/2}}^{2}
&\lesssim
\left(|\chi_{_{\mathbf{0}}}|+|\Psi_{_{\mathbf{0}}}|\right)^{2}
+
\sum_{\substack{k\in\mathscr{L}^{\star}, 
\\|\mathbf{k}|\neq \omega,0}}
|\mathbf{k}|^{2s+1}
\frac{\left(|\mathbf{k}|\tanh \dfrac{|\mathbf{k}|}{\beta_{0}}\right)^{2}}{(\det L_{0}(|\mathbf{k}|))^{2}}
\left( 
|\mathbf{k}|\tanh \dfrac{|\mathbf{k}|}{\beta_{0}}\
|\chi_{_{\mathbf{k}}}|
+|\Psi^\prime_{_{\mathbf{k}}}|
+|\Psi_{_{\mathbf{k}}}|
 \right)^{2}
\\
&\lesssim
|\chi_{_{\mathbf{0}}}|^{2}+|\Psi_{_{\mathbf{0}}}|^{2}
+
\sum_{\substack{k\in\mathscr{L}^{\star}, 
\\|\mathbf{k}|\neq \omega,0}}
|\mathbf{k}|^{2s-3}
\left( 
|\mathbf{k}|^{2}
|\chi_{_{\mathbf{k}}}|^{2}
+|\Psi^{\prime}_{_{\mathbf{k}}}|^2
+|\Psi_{_{\mathbf{k}}}|^{2}
 \right)
\\
&\lesssim
\left\lVert  \chi \right\rVert_{ s}^{2}
+\left\lVert  {{\Psi^\prime}} \right\rVert_{ s-1}^{2}
+\left\lVert  {{\Psi}} \right\rVert_{ s-{3/2}}^{2};
\end{align*}
similar calculations yield
$$
\left\lVert  \Phi^\prime \right\rVert_{ s}^{2}
\lesssim
\left\lVert  \chi \right\rVert_{ s}^{2}
+\left\lVert  {{\Psi^\prime}} \right\rVert_{ s-1}^{2}
+\left\lVert  {{\Psi}} \right\rVert_{ s-{3/2}}^{2}
,
\qquad
\left\lVert  \Phi \right\rVert_{ s}^{2}
\lesssim
\left\lVert  \chi \right\rVert_{ s}^{2}
+\left\lVert  {{\Psi^\prime}} \right\rVert_{ s-1}^{2}
+\left\lVert  {{\Psi}} \right\rVert_{ s-{3/2}}^{2}
.
$$
We conclude that $L_{0}^{-1}: Y_{\mathrm{r}}\to X_{\mathrm{r}}$ exists and is continuous.
\end{proof}
\begin{corollary}
The operator $L_{0}:X_0 \to Y_0$ is a Fredholm operator of index zero.
\end{corollary}
\begin{proof}
A straightforward calculation shows that $\mathbf{v} \notin\im{L_{0}(\omega)}$ 
(so that $\ker (L_{0}(\omega))^{2}=\ker L_{0}(\omega)$) 
and hence 
\begin{align*}
 Y_\omega
&=
\bigoplus_{\substack{\mathbf{k} \in\mathscr{L}^{\star}, \\
|\mathbf{k}|=\omega}}
\Big(\big( \im{L_{0}(\omega)}\oplus  \ker L_{0}(\omega)\big)\sin(\mathbf{k}\cdot\mathbf{x})
\oplus
\big(\im{L_{0}(\omega)}\oplus  \ker L_{0}(\omega)\big)\cos•\cdot\mathbf{x})\Big)
\\
&
=
\im{L_{0}}\big|_{ X_\omega}\oplus\ker L_{0}.
\end{align*}
Using this decomposition and Lemma \ref{lem: Lo}, we find that 
$$
Y_0 = Y_\omega \oplus Y_{\mathrm r} =\left (\im{L_{0}}\big|_{ X_\omega} \oplus \ker L_{0}\right) \oplus \im L_0\big|_{U_{\mathrm r}}
=
\im L_{0} \oplus \ker L_{0}.
$$
It follows that 
$\im L_{0}$ is closed and 
 $\mathrm{codim}\, \im L_{0}=\dim \ker L_{0} ,$ 
 so that $L_{0}:X_0 \to Y_0$ is a Fredholm operator of index zero. 
 \end{proof}

Because the kernel of $L_{0}$ is  multidimensional, we can not use Theorem \ref{thm: CRT} directly. 
To overcome this problem, we recall that ${\mathcal G}$ (and hence $L_0$)
is invariant under certain rotations (see below)
and seek solutions to \eqref{eq: Nfinal again}  in $X_0$ that have this rotational symmetry,
denoting the relevant  subspaces of $H^{r}_{\mathrm{per}}(\Gamma) ,
\bar{H}^{r}_{\mathrm{per}}(\Gamma),X_0$ and $Y_0$  
by $H^{r}_{\mathrm{sym}}(\Gamma),
\bar{H}^{r}_{\mathrm{sym}}(\Gamma), X_\mathrm{sym} $ and $Y_\mathrm{sym} , $ 
so that 
$$
X_\mathrm{sym}=
H^{s+{1/2}}_{\mathrm{sym}}(\Gamma)
\times
\bar{H}^{s}_{\mathrm{sym}}(\Gamma)
\times
H^{s}_{\mathrm{sym}}(\Gamma),
\qquad
Y_\mathrm{sym}=
H^{s}_{\mathrm{sym}}(\Gamma)
\times
\bar{H}^{s-1}_{\mathrm{sym}}(\Gamma)
\times
H^{s-{3/2}}_{\mathrm{sym}}(\Gamma)
$$
for $s>{5/2}$. Note that 
$X_\mathrm{sym}$ and $Y_\mathrm{sym}$ are invariant under $P_\omega$, so that
according to the above analysis $L_{0}:X_\mathrm{sym} \to Y_\mathrm{sym}$ is a Fredholm operator of index zero and
$$Y_\mathrm{sym}=\im L_0 \oplus \ker L_0.$$

\begin{enumerate}
\item\vspace{-\baselineskip}
For rolls we consider functions that are independent of the $z$-coordinate and lie 
in the subspace
\begin{align*}
H^{r}_{\mathrm{sym}}(\Gamma)
=
\left\{ \zeta\in H^{r}_{\mathrm{per}}(\Gamma):\
\zeta\left( x \right)  =  \zeta\left( -x \right)
 \right\}
\end{align*}
of functions which are invariant under rotations through 
$\pi.$ 
\item
For rectangles 
we work with the subspace
\begin{align*}
H^{r}_{\mathrm{sym}}(\Gamma)
=
\left\{ \zeta\in H^{r}_{\mathrm{per}}(\Gamma):\
\zeta\left( x,z \right)  =  \zeta\left( z,-x \right)
 \right\}
\end{align*}
of functions which are invariant under rotations through 
${\pi/2}.$  
\item
For hexagons 
we work with the subspace
\begin{align*}
H^{r}_{\mathrm{sym}}(\Gamma)
=
\left\{ \zeta\in H^{r}_{\mathrm{per}}(\Gamma):\
\zeta\left( x,z \right)=
\zeta\left( \frac{1}{2}
\left( x-\sqrt{3}z \right),
\frac{1}{2}\left( \sqrt{3}x+z \right) \right)
 \right\}
\end{align*}
of functions which are invariant under rotations through 
${\pi/3}.$  
\end{enumerate}
These restrictions ensure that $\dim \ker L_{0}=1$ 
with $\ker L_{0}=\langle  v_{0} \rangle,$ 
where  $v_{0}=\mathbf{v} e_{1}(x,z)$ 
with
$$e_1(x,z) = \left\{ \begin{array}{ll} \cos\omega x & \qquad \mbox{(rolls)} \\
\cos\omega x+\cos\omega z & \qquad \mbox{(rectangles)} \\
\cos \omega x
+\cos   \frac{\omega}{2}
\left(x+\sqrt{3}z\right)
+\cos \frac{\omega}{2}
\left(x-\sqrt{3}z\right)  & \qquad \mbox{(hexagons).}
\end{array}\right.$$
The projection 
$P:Y_\mathrm{sym}\to Y_\mathrm{sym}$ onto $\ker L_{0}$ along $\im L_{0}$ is given by 
\begin{equation}
P(\chi, \Psi^\prime, \Psi)^\mathrm{T}
=
{C^{\star}}((\chi_{_{(\omega,0)}}, \Psi^\prime_{_{(\omega,0)}}, \Psi_{_{(\omega,0)}})^{\mathrm{T}}\cdot\mathbf{v}^{\star})\mathbf{v}\, e_1(x,z),
\label{eq: formula for P}
\end{equation}
where

$$
\mathbf{v}^\star=\left(\begin{array}{c} 
\cfrac{\gamma_0+\omega^2}{\mu_1-1} \\
\cfrac{\gamma_0+\omega^2}{(\mu_1-1)^2\omega}
\left(S_1^{-1}\coth S_1\tilde{\omega}+\coth \tilde{\omega}\right) \\
1\\
\end{array}
\right),
$$
$$
C^{\star}=
\mu_{1}\omega
S_{1}^{-1}\tanh S_1\tilde{\omega}
-
\mu_{1}^{2}
S_{1}^{-1} 
\dfrac{
\tanh S_1 \tilde{\omega}
\big(\coth\tilde{\omega}
+ S_{1}\coth S_1 \tilde{\omega}\big)	
}{	
\tanh\tilde{\omega}\big(
\mu_{1} \coth\tilde{\omega}
+ S_{1}\coth S_1 \tilde{\omega}\big)}
+1
$$
($\mathbf{v}^\star\in\mathbb{R}^{3}$  solves the equation $L_{0}(\omega)^{\mathrm{T}}\mathbf{v}^\star=\mathbf{0}$
and $C^{\star}=(\mathbf{v}\cdot\mathbf{v}^\star)^{-1}$). 
\begin{lemma}\label{lem: Trans}
The transversality condition $P(\mathrm{d}_{1}\mathrm{d}_{2}{\mathcal G}[\gamma_0,0](1,v_{0}))\neq 0$ is satisfied.
\end{lemma}
\begin{proof}
It follows from the calculation
$
\mathrm{d}_{1}\mathrm{d}_{2}{\mathcal G}[\gamma_0,0](1,(\eta,\Phi^\prime,\Phi)^\mathrm{T})
=(0,0,-\eta)^\mathrm{T}$ and the formula \eqref{eq: formula for P} for $P$
that
$$\hspace{0.35cm}P(\mathrm{d}_{1}\mathrm{d}_{2}{\mathcal G}[\gamma_0,0](1,v_{0}))
=-C^{\star}(\mu_1-1)^{-1}S_1^{-1}\tanh S_1\tilde{\omega}(\mu_1\coth\tilde{\omega}+S_1\coth S_1\tilde{\omega})\mathbf{v}\, e_1(x,z).\hspace{0.32cm}\qedhere$$
\end{proof}
The facts established 
above 
confirm that the hypotheses of 
Theorem \ref{thm: CRT} 
are satisfied, 
an application of which yields Theorem \ref{thm: main result 1}.

\section{The bifurcating solution branches} \label{ch: trans, sub or super}

In this section we examine the bifurcating solution branches 
identified in Theorem \ref{thm: main result 1} by applying the following
supplement to the Crandall-Rabinowitz theorem.

\begin{theorem}\label{thm: CRT branches}

Suppose that the hypotheses of Theorem \ref{thm: CRT} hold. In the
notation of that theorem, 
let $Q:X\rightarrow  X$ be a projection with $\im Q = \ker L$ 
and 
the Taylor series of 
the functions \linebreak $w:(-\varepsilon,\varepsilon)\rightarrow V,$ $ \lambda: (-\varepsilon,\varepsilon)\rightarrow \mathbb{R}$ be given by 
$$
\lambda(s) = \lambda_{{{0}}}+s \lambda_{{{1}}}+s^{2}\lambda_{{{2}}}+\ldots, \qquad
w(s) =  v_{0}+s w_{1}+\ldots,
$$
where $\lambda_{{{1}}},\lambda_{{{2}}},\ldots\in \mathbb{R}$ and 
$w_{1},w_{2},\ldots\in \ker Q.$ 
\begin{enumerate}[label=(\roman{enumi})]
\item
The coefficient $\lambda_{{{1}}}$ satisfies the equation 
$$
P\left(\frac{1}{2!}{\mathrm{d}_{2}^{2}{\mathcal F}[\lambda_{{{0}}},0]({v_{0},v_{0}}})\right)
+\lambda_{{{1}}}P(\mathrm{d}_{1}\mathrm{d}_{2}{\mathcal F}[\lambda_{{{0}}},0](1,v_{0}))
=0
$$
and the bifurcation is transcritical if $\lambda_{{{1}}}$ is non-zero (see Figure \ref{pic: CR trans}). 
\begin{figure}[H]
\centering\includegraphics{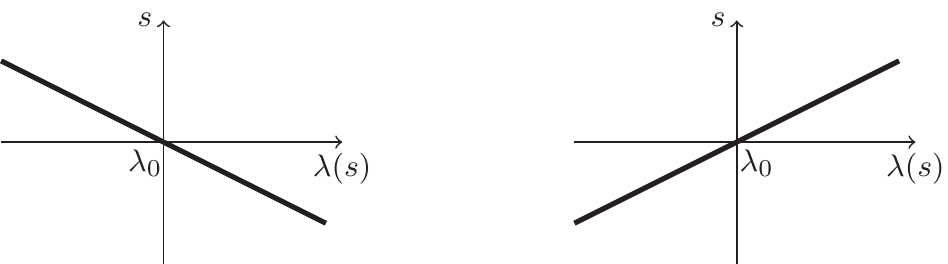}
\caption{\label{pic: CR trans}
Transcritical Crandall-Rabinowitz bifurcation for $\lambda_1<0$ (left) and $\lambda_1>0$ (right).} 
\end{figure}

\item
Suppose that $\lambda_{{{1}}}$ is zero. The coefficient $\lambda_{{{2}}}$ satisfies the equation 
$$
P\left( {\mathrm{d}_{2}^{2}{\mathcal F}[\lambda_{{{0}}},0](v_{0},w_{1})}
+\frac{1}{3!}\mathrm{d}_{2}^{3}{\mathcal F}[\lambda_{{{0}}},0](v_{0},v_{0},v_{0}) \right)+
\lambda_{{{2}}}P\left( \mathrm{d}_{1}\mathrm{d}_{2}{\mathcal F}[\lambda_{{{0}}},0](1,v_{0}) \right)
=0,
$$
where $w_{1}\in  \ker Q$ solves the equation 
$$
\mathrm{d}_{2}{\mathcal F}[\lambda_{{{0}}},0](w_{1})
=
-\frac{1}{2!}\mathrm{d}_{2}^{2}{\mathcal F}[\lambda_{{{0}}},0](v_{0},v_{0}).
$$
The bifurcation is {supercritical} for $\lambda_{{{2}}}>0$ 
and  {subcritical}  for $\lambda_{{{2}}}<0$ (see Figure \ref{pic: CR supsub}). 
\begin{figure}[H]
\centering\includegraphics{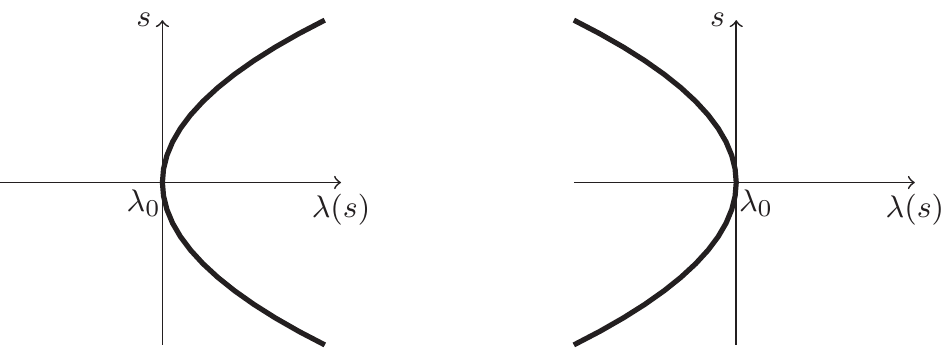}
\caption{\label{pic: CR supsub}
Crandall-Rabinowitz bifurcation for $\lambda_{{{1}}}=0$, $\lambda_{{{2}}}>0 $ (supercritical, left) and $\lambda_{{{1}}}=0$, $\lambda_{{{2}}}<0$ (subcritical,right).} 
\end{figure}
\end{enumerate}
\end{theorem}

To apply this theorem we write the Taylor series of the analytic functions
$w:(-\varepsilon,\varepsilon)\rightarrow V_\mathrm{sym}$,
$ \gamma: (-\varepsilon,\varepsilon)\rightarrow \mathbb{R}$ given in
Theorem \ref{thm: main result 1} as 
$$
\gamma(s) = \gamma_{0}+s \gamma_{1}+s^{2}\gamma_{2}+\ldots, \qquad
w(s) =  v_{0}+sw_{1}+\ldots
$$
with $\gamma_{1},\gamma_{2},\ldots\in\mathbb{R}$ and $w_{1},w_{2},\ldots\in \ker (I-P)$ 
and introduce the operators  
$$
L_{1}=\mathrm{d}_{1}\mathrm{d}_{2}{\mathcal G}[\gamma_0,0](1,\cdot),
\qquad
Q_{0}=\frac{1}{2!}\mathrm{d}_{2}^{2}{\mathcal G}[\gamma_0,0],
\qquad
C_{0}=\frac{1}{3!}\mathrm{d}_{2}^{3}{\mathcal G}[\gamma_0,0].
$$
One finds that $L_1(v)=(0,0,-\eta)^\mathrm{T}$ and
\begin{align*}
	Q_0&(v,v)
	\\
	&=
	\begin{pmatrix}
	0
	\\[1em]
	G_{2}^\prime  
	+ G_{2} +\dot{\mu}_{1} H_{2} 
	+\tfrac{1}{2}\left(\ddot{\mu}_{1} H_{1}^{2}+\dot{\mu}_{1}|\nabla\Phi|^{2}\right)
	\\[1em]
	\tfrac{1}{2}
	\left(
	G^{\prime \, 2}_{1}\!-|\nabla\Phi^\prime|^{2}
	+(\mu_{1}-\dot{\mu}_{1}) H_{1}^{2}
	+\mu_{1}|\nabla\Phi|^{2}
	-2 G_{1} H_{1} 
	-\ddot{\mu}_{1} H_{1}^{2}-\dot{\mu}_{1}|\nabla\Phi|^{2}\right)
	\\
	-\mu_{1} G_{2}^\prime  
	-G_{2} -\dot{\mu}_{1} H_{2} 
\end{pmatrix}
\end{align*}
\begin{align*}
	C_0&(v,v,v)
	\\
	& =	\begin{pmatrix}
	0
	\\[1.5em]
	G_{3}^\prime 
	+ 
	G_{3} +\dot{\mu}_{1} H_{3} 
	+\ddot{\mu}_{1} H_{1} H_{2} 
	+\tfrac{1}{2}\left(\ddot{\mu}_{1}-\dot{\mu}_{1}\right) |\nabla\Phi|^{2}H_{1} 
	-\dot{\mu}_{1}\nabla\eta\cdot\nabla\Phi \ H_{1} 
	+\tfrac{1}{6}\dddot{\mu}_{1} H_{1}^{3}
	\\[1.5em]
	-
	\mu_{1} G_{3}^\prime 
	-
	G_{3} -\dot{\mu}_{1} H_{3} 
	-\ddot{\mu}_{1} H_{1} H_{2} 
	-\tfrac{1}{2}\left(\ddot{\mu}_{1}-\dot{\mu}_{1}\right) |\nabla\Phi|^{2}H_{1} 
	+\dot{\mu}_{1}\nabla\eta\cdot\nabla\Phi \ H_{1} 
	-\tfrac{1}{6}\dddot{\mu}_{1} H_{1}^{3}
	\\[0.25em]
	+
G_{1}^\prime (
G_{2}^\prime -\nabla\eta\cdot\nabla\Phi^\prime
)
-H_{1} (G_{2} +\mu_{1}\nabla\eta\cdot\nabla\Phi)
	+(\mu_{1}-\dot{\mu}_{1})H_{1} H_{2} 
	+
	\tfrac{1}{3}
	(\dot{\mu}_{1}-\ddot{\mu}_{1})
	H_{1}^{3} 
	\\[0.25em]
	-G_{1} H_{2}+
	\eta_{x}^{2}\eta_{zz}+\eta_{z}^{2}\eta_{xx}
	-2\eta_{x}\eta_{z}\eta_{xz}
	-\tfrac{3}{2}|\nabla\eta|^{2}\eta
\end{pmatrix}
,	
\end{align*}
where  $v=(\eta, \Phi^\prime, \Phi)^\mathrm{T}$ and $\dddot{\mu}_1=\dddot{\mu}(1)$.
Theorem \ref{thm: CRT branches} shows that
$$
\gamma_{1}
=
-\frac{\left[{Q_{0}(v_{0},v_{0})}\right]_{1}\cdot\mathbf{v}^{\star}
}{
\left[{L_{1}v_{0}}\right]_{1}\cdot\mathbf{v}^{\star},
}
$$
where
$$\left[\zeta\right]_{1}={\zeta}_{_{(\omega,0)}}=\int_\Gamma \zeta e_1 $$
(with componentwise extension),
and
\begin{align*}
\gamma_{2}
&=
-\frac{\left[{2Q_{0}(v_{0},w_{1})
+C_{0}(v_{0},v_{0},v_{0})}\right]_{1}\cdot\mathbf{v}^{\star}}
{\left[{L_{1}v_{0}}\right]_{1}\cdot\mathbf{v}^{\star}}
\end{align*}
for $\gamma_1=0$, where 
$w_{1}\in\ker (I-P)$ solves the equation
\begin{align*}
L_{0}w_{1}
&=
-Q_{0}(v_{0},v_{0}).
\end{align*}

A straightforward calculation shows $Q_0(v_0, v_0)$ can be written as a sum
in which each summand is a constant vector multiplied by either $e_1^2$
or $|\nabla e_1|^2$.
For hexagons we find that $\gamma_1$ generally does not vanish,
while 
for rolls
$$
\left[{e_{1}^{2}}\right]_{1}
= 
\left[{\frac{1}{2}
+\frac{1}{2} \cos 2 \omega x}\right]_{1}
=0
,
\qquad
\left[{|\nabla e_{1}|^{2}}\right]_{1}
= 
\left[{\frac{1}{2}
-\frac{1}{2}\cos 2 \omega x}\right]_{1}
=0,
$$
and for rectangles
\begin{align*}
\left[e_{1}^{2}\right]_{1}
&= 
\left[\frac{1}{2}
+
{\cos \omega (x+z)+\cos \omega (x-z)}
+\frac{1}{2}({\cos 2 \omega x+\cos 2\omega z)}\right]_{1}
=0,
\\
\left[|\nabla e_{1}|^{2}\right]_{1}
&= 
\left[1
-\frac{1}{2}({\cos 2 \omega x \ +\cos 2\omega z)}\right]_{1}
=0,
\end{align*}
so that  in both cases $\gamma_{1}=0$.
Attempting to compute explicit general expressions for $\gamma_2$ leads to unwieldy formulae
(it appears more appropriate to calculate them numerically for a specific choice of $\mu$, that is a
specific magnetisation law). Here we confine ourselves to stating the values of the coefficients for two particular special cases.
\begin{enumerate}
\item
\emph{Constant relative permeability $\mu$ (corresponding to a linear magnetisation law):}
We find that
\begin{align*}
\label{eq: H2rolls}
\gamma_{2}
&
=
-\mu\frac{\mu-1}{\mu+1} \omega^{2}
\Bigg(\!\!\!-\mu C_{2}
\bigg(\bigg(8(\mu+1)\omega^{2}\frac{4\omega^{2}+\gamma_{0}}{t_{2}}-4(\mu^{2}-1)^{2}\omega^{3}\bigg)(1-t_{1}t_{2})^{2}
\nonumber\\
&
\hspace{1.35in}\mbox{}
+\omega^{3}(\mu-1)^{4}
\bigg(
2(
1+t_{1}^{2})(1-t_{1}t_{2})
-\frac{1}{4}(1+t_{1}^{2})^{2}
\bigg)\bigg)
\nonumber\\
&\hspace{0.85in}\mbox{}
-\frac{\mu\omega^{2}}{4\gamma_{0}}\frac{(\mu-1)^{3}}{\mu+1}(1-t_{1}^{2})^{2}
-
\frac{(\mu+1)^{2}\omega}{\mu-1}t_{1}
\bigg(
\frac{3\ \omega^{2}}{8(\gamma_{0}+\omega^{2})}
-\frac{3}{2}+t_{1}t_{2}
\bigg)\!\!\Bigg)
\end{align*}
for rolls and
\begin{align*}
\gamma_{2}
&=
-\mu\frac{\mu-1}{\mu+1} \omega^{2}
\Bigg(\!\!\!
-\mu C_{\sqrt{2}}
\bigg(\bigg(8(\mu+1)\omega^{2}\frac{2\omega^{2}+\gamma_{0}}{t_{\sqrt{2}}}-2\sqrt{2}(\mu^{2}-1)^{2}\omega^{3}\bigg)(1-\sqrt{2}t_{1}t_{\sqrt{2}})^{2}
\nonumber\\
&
\hspace{1.45in}\mbox{}
-\frac{\omega^{3}(\mu-1)^{4}}{2}\sqrt{2}\bigg(t_{1}^{4}-4t_{1}^{2}(1-\sqrt{2}t_{1}t_{\sqrt{2}})\bigg)\bigg)
\nonumber\\
&
\hspace{0.85in}
\mbox{}-\mu C_{2}
\bigg(\bigg(8(\mu+1)\omega^{2}\frac{4\omega^{2}+\gamma_{0}}{t_{2}}-4(\mu^{2}-1)^{2}\omega^{3}\bigg)(1-t_{1}t_{2})^{2}
\nonumber\\
&
\hspace{1.35in}\mbox{}
+\omega^{3}(\mu-1)^{4}
\bigg(
2(
1+t_{1}^{2})(1-t_{1}t_{2})
-\frac{1}{4}(1+t_{1}^{2})^{2}
\bigg)\bigg)
\nonumber\\
&\hspace{0.85in}
\mbox{}
-\frac{\omega (\mu+1)^{2}}{2(\mu-1)}t_{1}
\bigg(
\frac{5\omega^{2}}{4(\gamma_{0}+\omega^{2})}\!-\!9\!+2t_{1}t_{2}\!+\!4\sqrt{2}t_{1}t_{\sqrt{2}}
\bigg)\\
&\hspace{0.85in}
-\frac{\omega^{2}\mu(\mu-1)^{3}}{2\gamma_{0}(\mu+1)}(1-t_{1}^{2})^{2}\Bigg)
\end{align*} 
for rectangles, where $t_1=\tanh \tilde{\omega}$, $t_{_{\sqrt{2}}}=\tanh \sqrt{2}\tilde{\omega}$,
$t_2 = \tanh 2\tilde{\omega}$, $\tilde{\omega}=\omega/\beta_0$ and
\begin{align*}
C_{_{\sqrt{2}}}
&=
\frac{1}{\sqrt{2}(\mu^{2}-1)\omega}
\left(\sqrt{2}(\omega^{2}+\gamma_{0})\frac{t_{\sqrt{2}}}{t_{1}}-\gamma_{0}-2\omega^{2}\right)^{-1}, \\
C_2
&=
\frac{1}{2(\mu^{2}-1)\omega}
\left( 
2( \gamma_{0}+\omega^{2} )
\frac{t_2}{t_1}
-\gamma_{0}
-4\omega^{2}
 \right)^{-1}
\end{align*}
The sign of $\gamma_2$ clearly depends upon $\mu$ and $\tilde{\omega}$
(see Figure \ref{fig: c3 sign 1}).

\begin{figure}[h]
\centering
\includegraphics[scale=0.3]{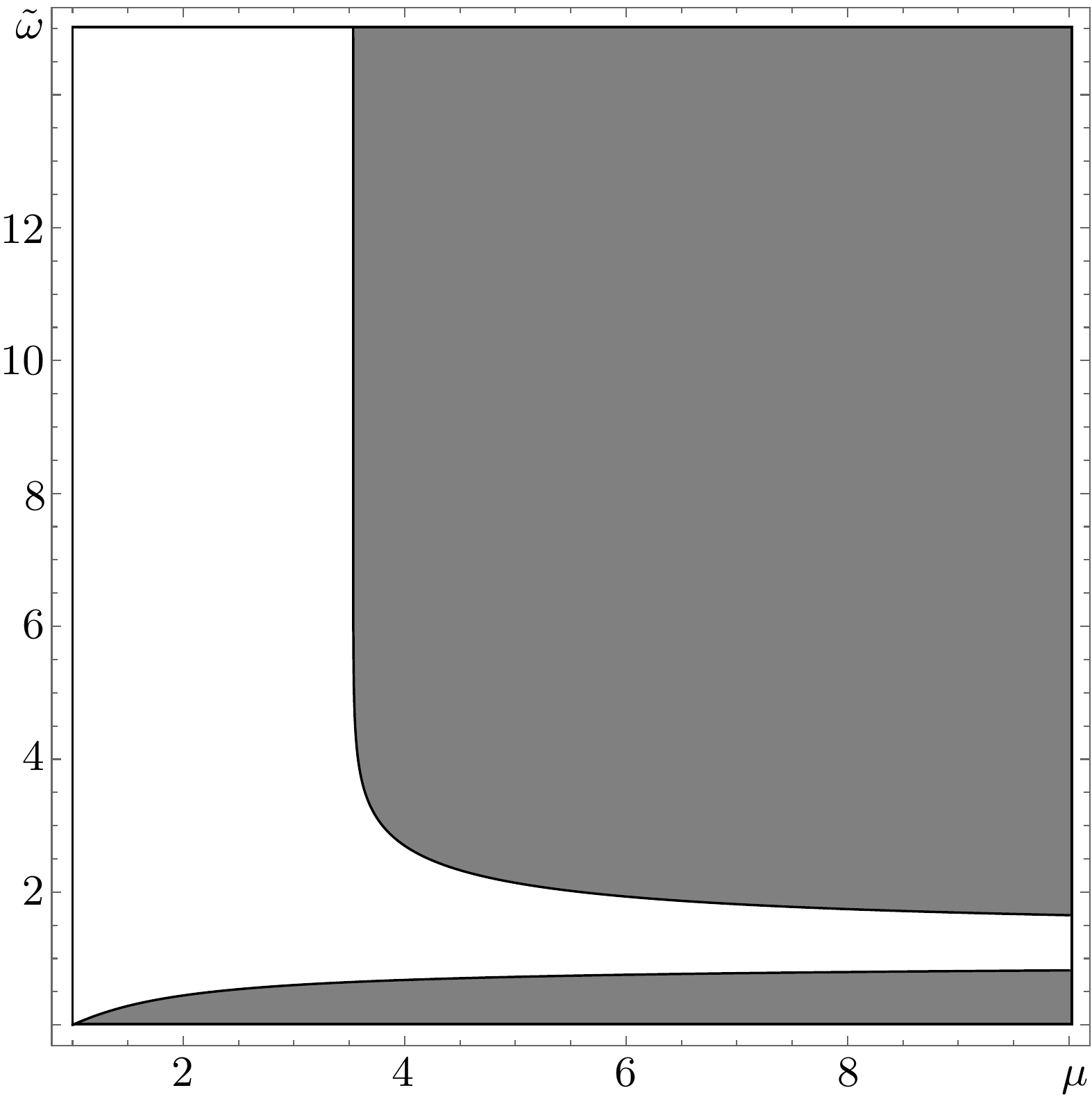}\hspace{1cm}\includegraphics[scale=0.3]{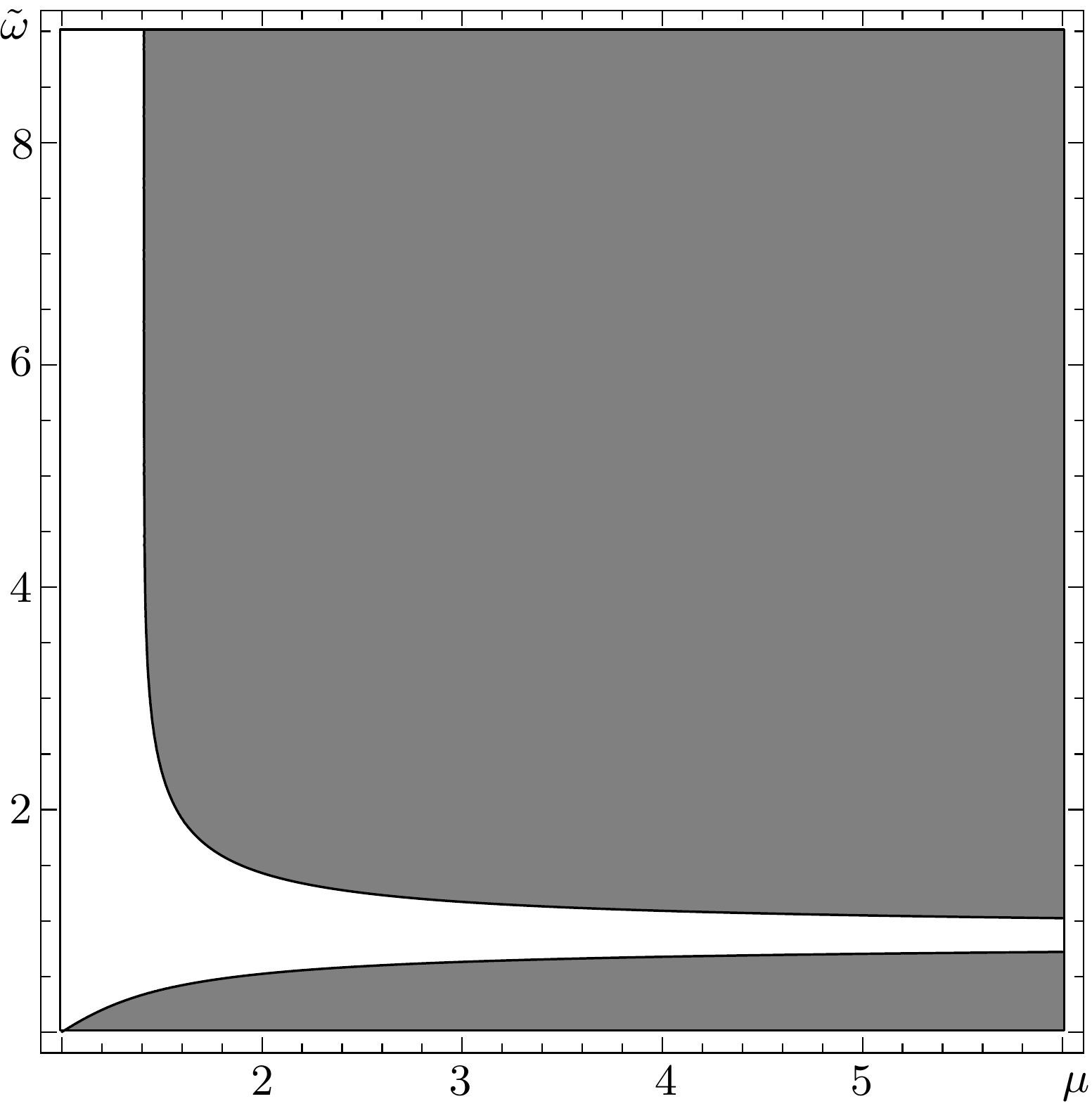}
{\it \caption{The sign of the coefficient $\gamma_2$ as a function of $\mu$ and $\tilde{\omega}$ for a linear magnetisation law for rolls (left) and rectangles (right). The shaded and white areas show the regions in which the bifurcation is respectively super- and subcritical.}
\label{fig: c3 sign 1}}
\end{figure}
\item
\emph{Small values of $\beta_{0}$ (corresponding to deep fluids):} Abbreviating 
$\mu(1),\dot{\mu}(1),\ddot{\mu}(1),\dddot{\mu}(1)$ to respectively $\mu_{1},\dot{\mu}_{1},\ddot{\mu}_{1},\dddot{\mu}_{1}$, one finds that
		{\scriptsize
		\begin{align*}
			\gamma_{2}
			&=
			\Big(
						t\big(
2  \mu_{1}^{3} \dddot{\mu}_{1}+2  \mu_{1}^{2} \dot{\mu}_{1} \dddot{\mu}_{1}-2  \mu_{1}^{2} \dddot{\mu}_{1}-2 \mu_{1}  \dot{\mu}_{1} \dddot{\mu}_{1}-\tfrac{11}{2}  \mu_{1}^{2}{ \ddot{\mu}_{1}}^{2}+\tfrac{11}{2}\mu_{1}  \ddot{\mu}_{1}^{2}+42  \mu_{1}^{3} \ddot{\mu}_{1}+49  \mu_{1}^{2} \dot{\mu}_{1} \ddot{\mu}_{1}-8  \mu_{1}^{2} \dot{\mu}_{1}
						\\
			&
			\qquad\qquad\mbox{}
			+29  \mu_{1}{ \dot{\mu}_{1}}^{2} \ddot{\mu}_{1}
			-10  \mu_{1}^{2} \ddot{\mu}_{1}+15 \mu_{1}
 \dot{\mu}_{1} \ddot{\mu}_{1}+3 { \dot{\mu}_{1}}^{2} \ddot{\mu}_{1}+8  \dot{\mu}_{1} \mu_{1}^{3}
+16  \mu_{1}^{2}{ \dot{\mu}_{1}}^{2}+8  \mu_{1} { \dot{\mu}_{1}}^{3}-16  \mu_{1} { \dot{\mu}_{1}}^{2}-8 { \dot{\mu}_{1}}^{3}
			\big)
						\\
			&
			\qquad\mbox{}
+
16  \mu_{1}^{4} \ddot{\mu}_{1}+32  \mu_{1}^{3} \dot{\mu}_{1} \ddot{\mu}_{1}+16  \mu_{1}^{2} \dot{\mu}_{1}^{2} \ddot{\mu}_{1}
+30  \mu_{1}^{2} \dot{\mu}_{1} \ddot{\mu}_{1}+5  \mu_{1}^{2}{ \ddot{\mu}_{1}}^{2}+10 \mu_{1} { \dot{\mu}_{1}}^{2} \ddot{\mu}_{1}
-16  \mu_{1}^{2} \ddot{\mu}_{1}-46 \mu_{1}  \dot{\mu}_{1} \ddot{\mu}_{1}
			\\
			&
			\qquad\mbox{}
-5 \mu_{1} \ddot{\mu}_{1}^{2}
-10 { \dot{\mu}_{1}}^{2} \ddot{\mu}_{1}
			\Big)\frac{(\mu_{1}-1)^{7}\mu_{1}^{2}}{512t^{6}(\mu_{1}+\dot{\mu}_{1})}
			\\
			&
			\quad
			+
			\Big(
			t\big(
96  \mu_{1}^{6}+636  \mu_{1}^{5} \dot{\mu}_{1}+1129  \mu_{1}^{4}{ \dot{\mu}_{1}}^{2}
+850  \mu_{1}^{3}{ \dot{\mu}_{1}}^{3}+217  \mu_{1}^{2}{ \dot{\mu}_{1}}^{4}-512  \mu_{1}^{5}-1720  \mu_{1}^{4} \dot{\mu}_{1}-1810  \mu_{1}^{3}{ \dot{\mu}_{1}}^{2}
			\\
			&			
			\qquad \qquad\quad\mbox{}
-740  \mu_{1}^{2}{ \dot{\mu}_{1}}^{3}-50  \mu_{1} { \dot{\mu}_{1}}^{4}+96  \mu_{1}^{4}+220  \dot{\mu}_{1} \mu_{1}^{3}+73  \mu_{1}^{2}{ \dot{\mu}_{1}}^{2}
-14 \mu_{1} { \dot{\mu}_{1}}^{3}-7{ \dot{\mu}_{1}}^{4}
			\big)	\\
			&\qquad\quad\mbox{}					
+88  \mu_{1}^{8}+440  \mu_{1}^{7} \dot{\mu}_{1}
+880  \mu_{1}^{6}{ \dot{\mu}_{1}}^{2}+880  \mu_{1}^{5}{ \dot{\mu}_{1}}^{3}
+440  \mu_{1}^{4}{ \dot{\mu}_{1}}^{4}+88  \mu_{1}^{3}{ \dot{\mu}_{1}}^{5}-\!256\!  \mu_{1}^{7}\!-\!928  \mu_{1}^{6} \dot{\mu}_{1}\!-\!1312  \mu_{1}^{5}{ \dot{\mu}_{1}}^{2}
			\\
			&
			\qquad \quad\mbox{}-864  \mu_{1}^{4}{ \dot{\mu}_{1}}^{3}
			-224  \mu_{1}^{3}{ \dot{\mu}_{1}}^{4}-80  \mu_{1}^{6}-416  \mu_{1}^{5} \dot{\mu}_{1}-614  \mu_{1}^{4}{ \dot{\mu}_{1}}^{2}
-420  \mu_{1}^{3}{ \dot{\mu}_{1}}^{3}-102  \mu_{1}^{2}{ \dot{\mu}_{1}}^{4}\!+\!256 \mu_{1}^{5}\!+\!672  \mu_{1}^{4} \dot{\mu}_{1}
			\\
			&
			\qquad \quad\mbox{}
-8  \mu_{1}^{4}+460  \mu_{1}^{3}{ \dot{\mu}_{1}}^{2}\!+\!104 \mu_{1}^{2}{ \dot{\mu}_{1}}^{3}-20  \mu_{1} { \dot{\mu}_{1}}^{4}\!+\!72  \dot{\mu}_{1} \mu_{1}^{3}
 \!+\!194  \mu_{1}^{2}{ \dot{\mu}_{1}}^{2}\!+\!84  \mu_{1} { \dot{\mu}_{1}}^{3}\!+\!10 { \dot{\mu}_{1}}^{4}
			\Big)\!\frac{(\mu_{1}-1)^{6}\mu_{1}}{1024t^{6}(\mu_{1}+\dot{\mu}_{1})}\\
			& \quad + o(1)
		\end{align*}
		}
as $\beta_0 \rightarrow 0$ for rolls and
{\scriptsize
		\begin{align*}
			{\gamma_{2}}
			&=
			\Big(
					t	\big(
(42 \sqrt{2}\!-\!60)\mu_{1}\dddot{\mu}_{1}\dot{\mu}_{1}
\!-\!(42 \sqrt{2}\!-\!60)\mu_{1}^{2} \dddot{\mu}_{1}   \dot{\mu}_{1}    
+(\!{\tfrac {411}{\sqrt{2}}} \!-\!289) \mu_{1} \ddot{\mu}_{1}^{2}
+\!(4461 \sqrt{2} \!-\!6190)  \ddot{\mu}_{1} \mu_{1}^{2}  \dot{\mu}_{1}  
  			\\
			&
			\qquad\qquad \mbox{}
+\!(3413  \sqrt{2}\!-\!4766 ) \ddot{\mu}_{1} \mu_{1}  \dot{\mu}_{1}  ^{2}
\!-\!(429\sqrt{2} \!-\!558)  \ddot{\mu}_{1} \mu_{1}  \dot{\mu}_{1}
+\!(168 \sqrt {2}    \!-\!240)   \dot{\mu}_{1}  ^{3}
+\!(168 \sqrt {2}\!-\!240) \mu_{1}^{2}  \dot{\mu}_{1}  
  			\\
			&
			\qquad\qquad \mbox{}
+\!(336 \sqrt {2}
\!-\!480 )\mu_{1}  \dot{\mu}_{1}  ^{2}
+\!(1870  \sqrt{2}
\!-\!2580  )\ddot{\mu}_{1} \mu_{1}^{3}
			+\!(82 \sqrt{2} \!-\!172  )\ddot{\mu}_{1} \mu_{1}^{2}
\!-\!\!(1333 \sqrt{2} \!-\!1886)  \ddot{\mu}_{1}   \dot{\mu}_{1}  ^{2}
  			\\
			&
			\qquad\qquad \mbox{}
+\!(42 \sqrt{2}\!-\!60) \mu_{1}^{2} \dddot{\mu}_{1} \!-\!(\!{\tfrac {411 }{\sqrt{2}}\!}\!-\!289) \mu_{1}^{2} \ddot{\mu}_{1} ^{2}
			\!-\!(168 \sqrt {2}\!-\!240)\mu_{1}  \dot{\mu}_{1}^{3}
\!-\!(168  \sqrt{2} \!-\!240)   \dot{\mu}_{1}  \mu_{1}^{3}
			  			\\
			&
			\qquad\qquad \mbox{}
\!-\!(336\sqrt {2}\!-\!480) \mu_{1}^{2}  \dot{\mu}_{1}  ^{2}
			\!-\!(42 \sqrt{2}\!-\!60) \mu_{1}^{3} \dddot{\mu}_{1} 
			\big)	
			\\
			&
			\qquad\mbox{}
			+\!(976 \sqrt {2}\!-\!1376) \mu_{1}^{4}\ddot{\mu}_{1}
			+\!(1952  \sqrt{2} \!-\!2752)   \dot{\mu}_{1}  \mu_{1}^{3}\ddot{\mu}_{1}
			+\!(976 \sqrt {2}\!-\!1376) \mu_{1}^{2}  \dot{\mu}_{1}  ^{2}			\ddot{\mu}_{1}
			+\!(445 \sqrt{2}\!-\!622) \mu_{1}^{2} \ddot{\mu}_{1} ^{2}			
			\\
			&
			\qquad\mbox{}
			+\!(366 \sqrt {2}\!-\!468 )\mu_{1}^{2}  \dot{\mu}_{1}  \ddot{\mu}_{1}
			\!-\!(1414 \sqrt {2}\!-\!2020) \mu_{1}  \dot{\mu}_{1}  ^{2}\ddot{\mu}_{1}
			\!-\!(445\sqrt{2}\!-\!622) \mu_{1}  \ddot{\mu}_{1} ^{2}
			\!-\!(976 \sqrt {2}\!-\!1376) \mu_{1}^{2}\ddot{\mu}_{1}
			\\
			&
			\qquad\mbox{}
			\!-\!(1342 \sqrt {2}\!-\!1844) \mu_{1}  \dot{\mu}_{1}  \ddot{\mu}_{1}
			+\!(1414 \sqrt {2}\!-\!2020)   \dot{\mu}_{1}  ^{2}		\ddot{\mu}_{1}
			\Big)\frac{(\mu_{1}\!-\!1)^{7}\mu_{1}^{2}}{512(10\!-\!7\sqrt{2})t^{6}(\mu_{1}+\dot{\mu}_{1})}
			\\
			&
			\quad
			+
			\Big(
			t \big(
			(3300\sqrt{2}\!-\!3288 )  \dot{\mu}_{1}  \mu_{1}^{3}
			+\!(5745 \sqrt {2}\!-\!7110 )\mu_{1}^{2}  \dot{\mu}_{1}  ^{2}
			+\!(2410 \sqrt {2}\!-\!3260) \mu_{1}{ \dot{\mu}_{1}  }^{3}
			\\
			&			
			\qquad \qquad\quad \mbox{}
			+\!(6084 \sqrt {2}	\!-\!6168) \mu_{1}^{5}  \dot{\mu}_{1} 
			+\!(7377 \sqrt {2}\!-\!6406 )\mu_{1}^{4}  \dot{\mu}_{1}  ^{2}
			\!-\!(3190 \sqrt {2}\!-\!7300) \mu_{1}^{3}  \dot{\mu}_{1}  ^{3}
			\\
			&			
			\qquad \qquad\quad \mbox{}
			-\!(6799 \sqrt {2}\!-\!10298) \mu_{1}^{2}  \dot{\mu}_{1}  ^{4}
			\!-\!(58696 \sqrt {2}\!-\!79024) \mu_{1}^{4}  \dot{\mu}_{1}  
			\!-\!(61154 \sqrt {2}\!-\!81804) \mu_{1}^{3}  \dot{\mu}_{1}  ^{2}
			\\
			&			
			\qquad \qquad\quad \mbox{}
			\!-\!(11604 \sqrt {2}\!-\!14456) \mu_{1}^{2}  \dot{\mu}_{1}  ^{3}
			+\!(11550 \sqrt {2}\!-\!16500) \mu_{1}  \dot{\mu}_{1}  ^{4}
			+\!(672 \sqrt {2}\!-\!448) \mu_{1}^{6}
			\\
			&			
			\qquad \qquad\quad \mbox{}
			\!-\!(17408 \sqrt {2}\!-\!23552) \mu_{1}^{5}
			\!-\!(2351 \sqrt {2}  \!-\!3322)   \dot{\mu}_{1}  ^{4}
			+\!(672 \sqrt{2}\!-\!448) \mu_{1}^{4}
			\big)
			\\
			&
			\qquad\quad \mbox{}
			+\!(2344\sqrt{2} \!-\!3056 )\mu_{1}^{3}  \dot{\mu}_{1}  ^{5}
			+\!(11720 \sqrt {2}\!-\!15280) \mu_{1}^{7}  \dot{\mu}_{1}  
			+\!(23440 \sqrt {2}\!-\!30560) \mu_{1}^{6}  \dot{\mu}_{1}  ^{2}
			\\
			&
			\qquad\quad \mbox{}
			+\!(23440 \sqrt {2}\!-\!30560 )\mu_{1}^{5}  \dot{\mu}_{1}  ^{3}
			+\!(11720 \sqrt {2}\!-\!15280 )\mu_{1}^{4}{\dot{\mu}_{1}}^{4}
			\!-\!(36384 \sqrt {2}\!-\!49344) \mu_{1}^{6}  \dot{\mu}_{1}  
			\\
			&
			\qquad\quad \mbox{}
			-\!(60832\sqrt {2}\!-\!82880) \mu_{1}^{5}  \dot{\mu}_{1}  ^{2}
			\!-\!(47328 \sqrt {2}\!-\!64832 )\mu_{1}^{4}  \dot{\mu}_{1}^{3}
			\!-\!(14176 \sqrt {2}\!-\!19520) \mu_{1}^{3}  \dot{\mu}_{1}  ^{4}
			\\
			&
			\qquad\quad \mbox{}
			+\!(4344 \sqrt{2}\!-\!6608)   \dot{\mu}_{1}  \mu_{1}^{3}
			\!-\!(478 \sqrt {2}\!-\!628)\mu_{1}^{2}  \dot{\mu}_{1}  ^{2}
			\!-\!(4860 \sqrt {2}\!-\!7016) \mu_{1}  \dot{\mu}_{1}  ^{3}
			+\!(2344 \sqrt {2}\!-\!3056) \mu_{1}^{8}
			\\
			&
			\qquad\quad \mbox{}
			\!-\!(8704 \sqrt {2}\!-\!11776) \mu_{1}^{7}
			\!-\!(15392 \sqrt {2}\!-\!21440) \mu_{1}^{5}  \dot{\mu}_{1}
			-\!(18422 \sqrt {2}\!-\!25348) \mu_{1}^{4}  \dot{\mu}_{1}  ^{2}
			\\
			&
			\qquad\quad \mbox{}
			\!-\!(8196 \sqrt {2}\!-\!10904) \mu_{1}^{3}  \dot{\mu}_{1}  ^{3}
			+\!(2410 \sqrt {2}\!-\!3644) \mu_{1}^{2}  \dot{\mu}_{1}  ^{4}
			+\!(27680 \sqrt {2}\!-\!37568) \mu_{1}^{4}  \dot{\mu}_{1} 
			\\
			&
			\qquad\quad \mbox{}
			+\!(33644 \sqrt {2}\!-\!46344) \mu_{1}^{3} \dot{\mu}_{1}  ^{2}
			+\!(18088 \sqrt {2}\!-\!25328 )\mu_{1}^{2}  \dot{\mu}_{1}  ^{3}
			-\!(3700 \sqrt {2}\!-\!5176) \mu_{1}  \dot{\mu}_{1}  ^{4}
			\\
			&
			\qquad\quad \mbox{}
			-\!(4016 \sqrt {2}\!-\!\!5664) \mu_{1}^{6}
			+\!(8704 \sqrt {2}\!-\!\!11776 )\mu_{1}^{5}
			+\!(1850 \sqrt {2}\!-\!\!2588)   \dot{\mu}_{1}  ^{4}
			\\
			&
			\qquad\quad \mbox{}
			+\!(1672 \sqrt {2}\!-\!\!2608 )\mu_{1}^{4}
			\Big)\frac{(\mu_{1}-1)^{6}\mu_{1}}{1024(10-7\sqrt{2})t^{6}(\mu_{1}+\dot{\mu}_{1})}+o(1)
\end{align*}
}
as $\beta_0 \rightarrow 0$ for rectangles, 	where $t=\sqrt{\mu_1(\mu_1+\dot{\mu}_1)}+1$.
(Note that
$$\omega = \frac{\mu_1(\mu_1-1)^2}{2(\mu_1+S_1)}+o(1), \qquad
\gamma_0 = \frac{\mu_1^2(\mu_1-1)^4}{4(\mu_1+S_1)^2}+o(1)$$
as $\beta_0 \rightarrow 0$.)

We note in particular that for constant $\mu$ (corresponding to a linear magnetisation law),
rolls bifurcate subcritically for $\mu<\mu_\mathrm{c}^1$ and
supercritically for $\mu>\mu_\mathrm{c}^1$, while rectangles bifurcate subcritically for
$\mu<\mu_\mathrm{c}^2$ and supercritically for $\mu>\mu_\mathrm{c}^2$, where
$$
\mu_\mathrm{c}^1=\frac{21}{11}+\frac{8}{11}\sqrt{5},
\qquad
\mu_\mathrm{c}^2=\frac{115+160\sqrt{2}+8\sqrt{184+11\sqrt{2}}}{141+128\sqrt{2}}.
$$

Figure \ref{fig: c2 sign 2} shows the sign of $\gamma_2$ for the Langevin magnetisation law
\begin{equation}
\mu(s)=1+\frac{M}{s}\left(\coth(\gamma s)-\frac{1}{\gamma s}\right)
\label{eq: Langevin}
\end{equation}
in the limit $\beta_0 \rightarrow 0$, where $M$ and $\chi_0$
are respectively the magnetic saturation and initial susceptibility of the ferrofluid
and $\gamma=3\chi_0/M$.

\begin{figure}[h]
\centering
\includegraphics[scale=0.3]{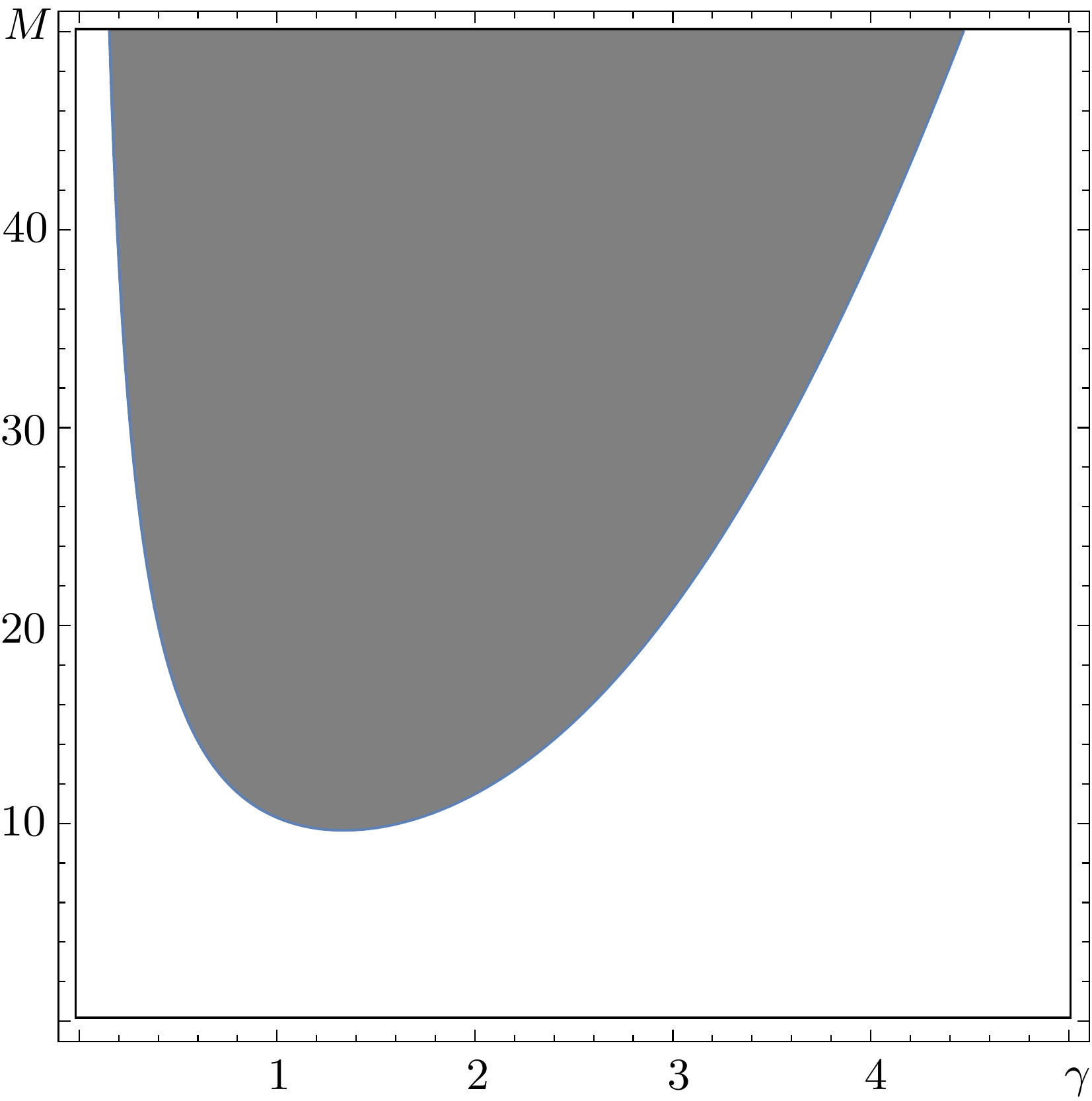}\hspace{1cm}\includegraphics[scale=0.3]{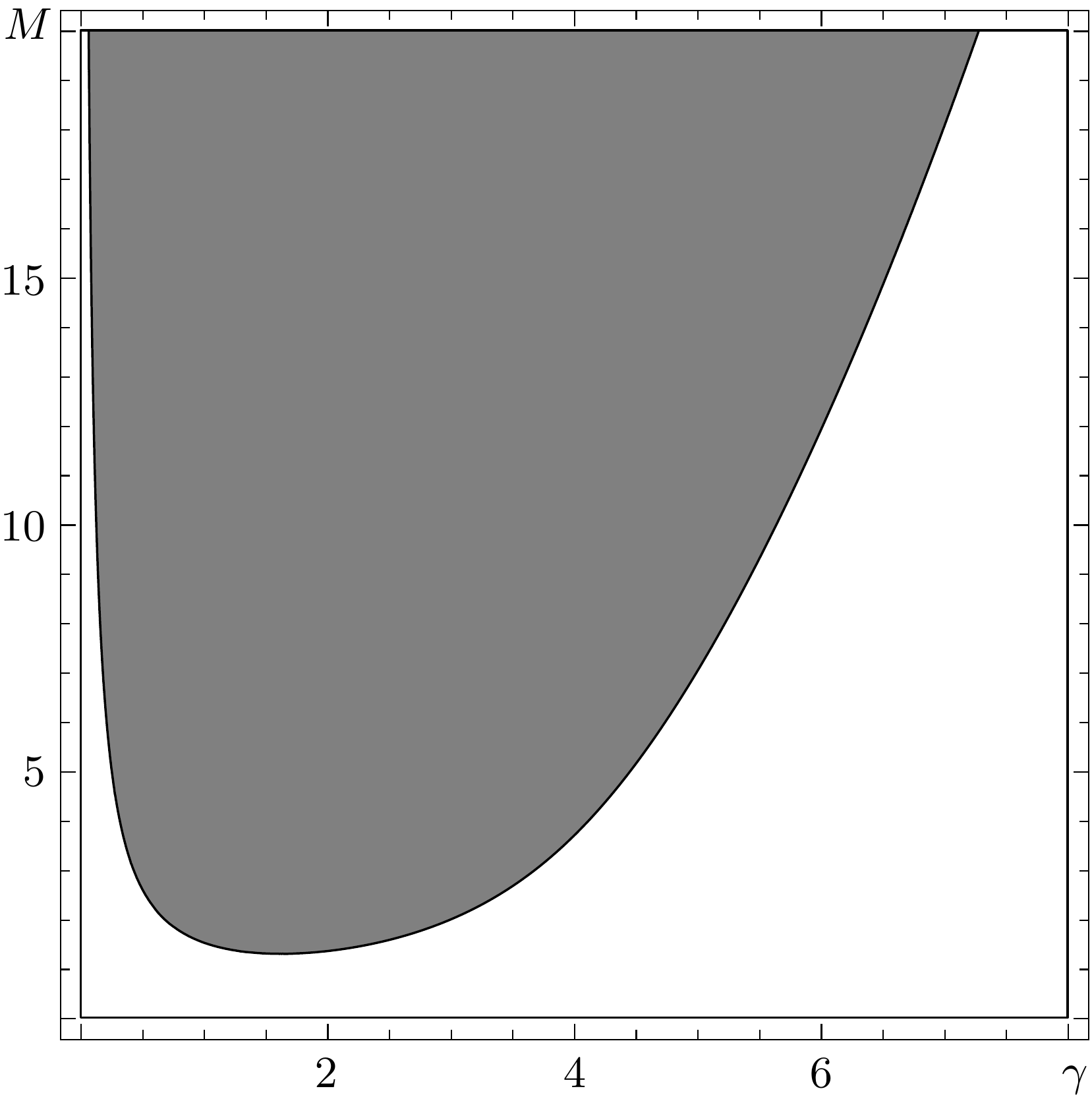}
{\it \caption{The sign of the coefficient $\gamma_2$ as a function of $M$ and $\gamma$ for
the Langevin magnetisation law \eqref{eq: Langevin} 
for rolls (left) and rectangles (right)
in a ferrofluid of great depth. The shaded and white areas show the regions in which the bifurcation is respectively super- and subcritical.}
\label{fig: c2 sign 2}}
\end{figure}
\end{enumerate}

\vspace{-12mm}
\dataccess{This paper has no additional data.}\vspace{-0.5\baselineskip}
\aucontribute{M.D.G. and J.H. collaborated equally on each aspect of this work.}\vspace{-0.5\baselineskip}
\competing{We declare we have no competing interests.}\vspace{-0.5\baselineskip}
\funding{No external funding was involved in this work.}


\vspace{-2mm}

\begin{thebibliography}{99}

\bibitem{CowleyRosensweig67}
Cowley MD, Rosensweig RE. 1967  The interfacial stability of a ferromagnetic
  fluid. {\em J. Fluid Mech.} \textbf{30}, 671--688.

\bibitem{TwomblyThomas83}
Twombly EE, Thomas JW. 1983  Bifurcating instability of the free surface of a
  ferrofluid. {\em SIAM J. Math.\ Anal.} \textbf{14}, 736--766.

\bibitem{CraigNicholls00}
Craig W, Nicholls DP. 2000  Traveling two and three dimensional capillary
  gravity water waves. {\em SIAM J. Math.\ Anal.} \textbf{32}, 323--359.

\bibitem{BuffoniToland}
Buffoni B, Toland JF. 2003 {\em Analytic Theory of Global Bifurcation}.
Princeton, N. J.: Princeton University Press.

\bibitem{Sattinger78}
Sattinger DH. 1978  Group representation theory, bifurcation theory and pattern
  formation. {\em J. Func.\ Anal.} \textbf{28}, 58--101.

\bibitem{SilberKnobloch88}
Silber M, Knobloch E. 1988  Pattern selection in ferrofluids. {\em Physica D}
  \textbf{30}, 83--98.

\bibitem{LloydGollwitzerRehbergRichter15}
Lloyd DJB, Gollwitzer C, Rehberg I, Richter R. 2015  Homoclinc snaking near the
  surface instability of a polarisable fluid. {\em J. Fluid Mech.}
  \textbf{783}, 283--305.

\bibitem{GrovesLloydStylianou17}
Groves MD, Lloyd DJB, Stylianou A. 2017  Pattern formation on the free surface
  of a ferrofluid: spatial dynamics and homoclinic bifurcation. {\em Physica D}
  \textbf{350}, 1--12.

\bibitem{Rosensweig}
Rosensweig RE. 1997 {\em Ferrohydrodynamics}.
New York: Dover.

\bibitem{NichollsReitich01a}
Nicholls DP, Reitich F. 2001  A new approach to analyticity of
  {D}irichlet-{N}eumann operators. {\em Proc.\ Roy.\ Soc.\ Edin.\ A}
  \textbf{131}, 1411--1433.

\end{thebibliography}
\end{document}